\documentclass[12pt]{amsart}
\usepackage{latexsym, amsfonts, amsmath, amscd, esint}
\usepackage{mathrsfs}
\usepackage[bookmarksnumbered,colorlinks,plainpages,backref]{hyperref}
\usepackage[lmargin=2.5cm,tmargin=2cm,bmargin=2cm,rmargin=2.5cm]{geometry}
\usepackage{amssymb}

\usepackage[dvips]{graphicx}
\usepackage{comment}
\usepackage{ulem}

\newtheorem{definition}{Definition}[section]
\newtheorem{proposition}{Proposition}[section]

\newtheorem{lemma}{Lemma}[section]
\newtheorem{remark}{Remark}[section]
\newtheorem{corollary}{Corollary}[section]

\numberwithin{equation}{section}

\newcommand\norma[1]{\left\lVert#1\right\rVert}

\newcounter{alphatheo}
\newenvironment{alphatheo}{\refstepcounter{alphatheo}\medskip\noindent{\bf Theorem\ \Alph{alphatheo}.} \it }{\medskip}

\newcounter{example}

\renewenvironment{proof}{\medskip\noindent{\sc Proof:}}{\medskip}

\def\qed{\ifhmode\unskip\nobreak\fi\quad 
  \ifmmode\square\else$\square$\fi}

\newcommand{\R}{\mathbb R}

\newcommand{\N}{\mathbb N}
\newcommand{\Z}{\mathbb Z}
\newcommand{\C}{\mathbb C}

\def\erre{\mathbb R}

\def\ze{\mathbb Z}
\def\ene{\mathbb N}

\def\bsL{{\mathcal L}}

\def\D{{\mathcal D}}
\def\E{{\mathcal E}}

\def\L{{\mathcal L}}
\def\M{{\mathcal M}}
\def\O{{\mathcal O}}
\def\S{{\mathcal S}}

\def\X{{\mathcal X}}

\def\ccinf{C^\infty_{c}}

\def\bmo{\hbox{\rm bmo\,}}

\def\<{\langle}
\def\>{\rangle}
\def\eps{\varepsilon}

\def\D{{\mathcal D}}
\def \P{{\mathcal P}}

\newcommand{\supp}{\text{supp}\,}

\newcommand{\diver}{{\rm div\,}}

\newcommand\inner[2]{\langle #1, #2 \rangle} 

\begin{document}

\title[Higher order div-curl type estimates for elliptic linear operators]{Higher order div-curl type estimates for elliptic homogeneous linear differential operators on localizable Hardy spaces}

\author {C. Machado}
\address{Instituto de Ci\^encias Matem\'aticas e de Computa\c{c}\~ao, Universidade de S\~ao Paulo, S\~ao Carlos, SP, 13566-590, Brasil}
\email{catarina.machado@usp.br}

\author {T. Picon}
\address{Departamento de Computa\c{c}\~ao e Matem\'atica, Universidade de S\~ao Paulo, Ribeir\~ao Preto, SP,  14040-901, Brasil}
\email{picon@ffclrp.usp.br}

\thanks{The authors were supported by Funda\c{c}\~ao de Amparo \`a  Pesquisa do Estado de S\~ao Paulo (FAPESP - grants 18/15484-7, 21/12655-8, 24/12753-8 and 25/00433-1) 
and the second by Conselho Nacional de Desenvolvimento Cient\'ifico e Tecnol\'ogico (CNPq - grants 315478/2021-7 and 302676/2025-2)}


\begin{abstract}
In this work, we establish higher-order div-curl type estimates in the sense of  Coifman, Lions, Meyer \& Semmes, in a local setting for elliptic homogeneous linear differential operators with smooth coefficients acting on localizable Hardy spaces. 
Our results imply and extend previously known estimates for first-order operators associated with elliptic systems and complexes of vector fields. As tools of independent interest, we develop a new smooth atomic decomposition for localizable Hardy-Sobolev spaces and prove a Poincar\'e-type inequality in this framework.

\end{abstract}

\maketitle

\section{Introduction}

In the classical work  due to Coifman, Lions, Meyer \& Semmes in \cite{CLMS}, some nonlinear inequalities were studied, in particular the so called div-curl estimates on Hardy spaces in $\R^{N}$. More precisely, if 
	\begin{equation}\label{parameters}
	\frac{N}{N+1}<p <\infty,\quad 1<q\leq \infty\quad  \text{and}\quad\frac{1}{r}:=\frac{1}{p}+\frac{1}{q}<1+\frac{1}{N}
	\end{equation}
then there exists a constant $C>0$ such that
\begin{equation}\label{eq01}
\|u \cdot v\|_{H^{r}}\leq C \|u\|_{H^{p}}\|v\|_{H^{q}},
\end{equation}
where $u \in H^{p}(\R^{N},\R^{N})$ and $v \in H^{q}(\R^{N},\R^{N})$ are vector fields satisfying ${\rm curl}\;u=0$ and ${\rm div}\;v=0$, respectively. Clearly, if ${\rm curl}\;u=0$ then we may write $u=\nabla \phi$ and since ${\rm div}\,v=0$, we have  
$u\cdot v =\nabla\phi\cdot v={\rm div}\,(\phi\, v)$. Under assumption ${\rm div}\,v=0$,  the estimate \eqref{eq01} may equivalently be written as
\begin{equation}\label{eq03}
\|\nabla\phi \cdot  v\|_{H^{r}}\leq C \|\nabla \phi\|_{H^{p}}\|v\|_{H^{q}}.
\end{equation} 
Note that this restriction can be interpreted as requiring $v$  to belong to the kernel of the formal adjoint of the gradient operator.
We point out that in the particular case $r=1$, the previous inequalities improve upon the information provided by H\"older's inequality. 
We recall that the Hardy space $H^{1}(\R^{N})$ is a strict subspace of  $L^{1}(\R^N)$. 

Estimates of this type have been extended in several directions as nonhomogeneous version for local Hardy spaces and BMO (\cite{CDY}, \cite{CGS}, 
\cite{D}), paraproducts (\cite{BGK}, \cite{BIJZ}, \cite{YYZ}),  Lipschitz domains in $\R^{N}$ (\cite{ART}, \cite{LM1}, \cite{LM2}), weighted Hardy spaces (\cite{BFG}), and applications to the Navier-Stokes equations (\cite{Miyakawa}).

In \cite{HHP}, the authors obtained a local version of inequality \eqref{eq03} in the setting of systems and complexes of vector fields with complex variable coefficients.
Suppose that $\L := \left\{L_{1},\dots,L_{n}\right\}$ is a  system of linearly independent vector fields with smooth complex coefficients defined on an open set $\Omega \subset \R^{N}$ and consider the operators 
$$\nabla_{\L}\,u := (L_{1}u,\dots,L_{n}u), \quad \text{for} \quad u\in C^{\infty}(\Omega)$$ 
and 
$${\rm div}_{\L^{*}}\,v := \sum_{j=1}^{n}L^{*}_{j}v_{j}, \quad \text{for} \quad  v\in C^{\infty}(\Omega,\R^{n}),$$ 
 which are precisely the operators $\nabla$ and $\rm{div}$ when $n=N$ and $L_{j}=\partial_{x_{j}}$ for $j=1,\dots,n$.
 Here  $L_{j}^{*}=\overline{ L_j^{t}}$, where $\bar L_{j}$ denotes the vector field obtained from $L_j$ by conjugating its coefficients and $L_j^t$ is the formal transpose of $L_j$. 
 Since the classical Hardy spaces $H^p(\erre^N)$ are not closed under multiplication by test functions, and the related operators have variable coefficient, it is appropriate to consider the nonhomogeneous version of Hardy spaces denoted by $h^p(\erre^N)$, also called localizable Hardy spaces, introduced by Goldberg in \cite{G}. 
 The \cite[Theorem A]{HHP} asserts that if $\mathcal{L}$ is an elliptic system of complex vector fields  defined on $\Omega\subset\R^N$ with $N \geq 2$ and  $r,p,q$ satisfying \eqref{parameters}
then for every point $x_{0} \in \Omega$ there exist an open neighborhood $x_0\in U \subset \Omega$ and $C(U)>0$
	such that 
	\begin{equation}\label{estimateHHP}
		\|\nabla_{\mathcal{L}}\phi\cdot v\|_{h^r} \le C\|\nabla_{\mathcal{L}}\phi\|_{h^p}\|v\|_{h^{q}},
	\end{equation}
	holds for any $\phi\in C_{c}^{\infty}(U)$ and $v\in C^{\infty}_{c}(U,\C^{n})$ satisfying ${\rm div}_{\mathcal{L^{\ast}}}\;v=0$ in the sense of distributions. 
The ellipticity of the system $\left\{L_{1},\dots,L_{n}\right\}$ means that, for any {\it real} 1-form $\omega$ satisfying  $\left\langle \omega, L_{j}\right\rangle=0$ for all $j=1,\dots,n$ implies $\omega=0$, that is equivalent to saying that the second order operator
\begin{equation*}
\Delta_{{\L}}\; :=\;L_1^* L_1+\cdots+L_n^* L_n
\end{equation*}
is elliptic in the classical sense. Extensions of the inequality \eqref{estimateHHP} for complexes associated to the system $\L$, inspired in the global setting for de Rham complex (see \cite{CLMS}), were also presented in \cite{HHP} (see the Section \ref{S7} below).  

A natural question arises concerning how to obtain analogous inequalities when the gradient is replaced by higher order homogeneous linear differential operators with variable coefficients. Partial results were obtained by Taylor \cite[Chapter 8, Section 8]{Tay} in the setting of pseudodifferential operators  and by Kozono \& Yanagisawa in \cite{KY} on Lebesgue spaces, including first order linear differential operators with smooth variable coefficients. Recall that the spaces $H^{p}(\R^{N}),h^{p}(\R^{N}),L^{p}(\R^{N})$ are equivalents with comparable norms, when $p>1$.  In this paper, 
our main goal is to address this question by establishing
a type of div-curl estimates for high order homogeneous linear differential operators with variable coefficients in Hardy spaces, in the same spirit as \eqref{eq03}. 
To obtain estimates analogous to \eqref{estimateHHP} for operators with variable coefficients, it is necessary to work with local rather than global estimates; for this reason, the use of localizable Hardy spaces is appropriate.

Let $\Omega \subseteq \R^{N}$ be an open set and let $A(\cdot,D)$ be a homogeneous linear differential operator of order $m$ with smooth complex coefficients on $\Omega$ given by
\[
A(x,D)=\sum_{|\alpha|{=} m} 
a_\alpha(x)\partial^\alpha: C^{\infty}(\Omega,E)\mapsto C^{\infty}(\Omega,F),
\]
where $E$ and $F$ are finite-dimensional complex vector spaces with $dim\;F\geq dim\;E$. 
Our main result is the following.

\begin{alphatheo}\label{teo_principal}
	 Let $A(\cdot,D)$ be an elliptic homogeneous linear differential operator with order m on $\Omega \subset \R^N$, and let $r, p,q$ satisfy
	 $$\dfrac{N}{N+m} < p \leq {\frac{N}{m}}, \; \; \; 1 < q \leq \infty \; \; \text{and} \;\; \dfrac{1}{r} \; := \; \dfrac{1}{p} + \dfrac{1}{q} < 1 + \dfrac{m}{N}.$$
	Then for each $x_0 \in \Omega$, there exists an open neighborhood $x_0 \in U \subset \Omega$ and a constant $C>0$ such that
\begin{equation}\label{maineqa}
		\norma{A(\cdot,D)\phi \cdot v}_{h^r} \leq C \norma{A(\cdot,D)\phi}_{h^p} \|v\|_{W^{m-1,q}} 
\end{equation}
	for any $\phi \in C^{\infty}_c(U,E)$ and $v \in C^{\infty}_c(U,F)$ satisfying
	$A^*(\cdot, D)v = 0$.
	
\end{alphatheo}

We say that $A(x,D)$ is elliptic in $\Omega$ if, for each $x_{0} \in \Omega$, its symbol  
$$a(x_{0},\xi):=\displaystyle \sum_{|\alpha|=m} a_{\alpha} (x_{0}) \xi^{\alpha}$$
{is injective}.  
Here $\displaystyle{\|v\|_{W^{k,q}}:=\sum_{|\gamma| \leq k}\norma{D^{\gamma}v}_{L^q}}$ is the Sobolev norm of order $k \in \N$.  We use the notation  $A^{*}:= \overline{ A^{t}}$, where $\overline A$ denotes the operator obtained from $A$ by conjugating its coefficients and $A^t$ denotes its formal transpose, it  means that for all smooth functions $\varphi$ and $\psi$ having compact support in $\Omega$ and taking values in $E$ and $F$ respectively, we have
 $$
 \int_\Omega A(x,D)\varphi \cdot_{F} {\psi}(x)dx = \int_\Omega \varphi(x) \cdot_{E} {A^*(x,D)\psi(x)}dx.
 $$

A nonhomogeneous version of the inequality \eqref{maineqa} is stated in Theorem B, where the condition $A^{\ast}(\cdot,D)v=0$ is avoided. The proof of this result and some applications extending and recovering the inequality for first order operators associated with the system of complex vector fields  in \cite{HHP} will be presented at Section \ref{S7}. Note that, the estimate \eqref{maineqa} reduces to \eqref{estimateHHP} when $A(\cdot,D)=\nabla_{\L}$, in which case $m=1$ and $A^{*}(\cdot,D)=\diver_{\L^{\ast}}$.  

{We point out that the critical point ${N}/{(N+m)}$ in Theorem A approaches to $0$ when the order of the operator $m$ is higher. It is well known that in Hardy spaces some cancellation conditions are required when $p \searrow 0$ (see \cite{D1,D2} for more details) that means the necessity of new treatment in comparison to the case $m=1$.} 
Moreover, allowing a much larger class of higher-order linear differential operators in  \eqref{maineqa} requires a substantial refinement of the proof strategy. The core of the proof of Theorem A relies on a Poincar\'e type inequality for localizable Hardy-Sobolev spaces, namely: for each $x \in \R^{N}$, $0<t<1$ and $f \in C_{c}^{\infty}(\R^N)$ there exist a special polynomial $P_{x,t,f}$ with degree less than $m-1$ and  a positive (universal) constant $C$ such that
\begin{align}\label{5.2a}
\left( \int_{\R^{N}} \bigg[ \sup_{0<t<1}\fint_{B_{x}^{t}} \bigg\{\frac{1}{t^m} |f(y)- P_{x,t,f} (y)| \bigg\}^{\alpha}dy \bigg]^{{p}/{\alpha}}dx\right)^{1/p} \leq {C \|f\|_{h^{m,p}}}, 
\end{align}
where  
$\alpha$ is a parameter depending on $m,N$ and $p$ to be chosen. 
The functional $\|f\|_{h^{m,p}}$ denotes the non-homogeneous Hardy-Sobolev quasi-norm given by $\sum_{|\beta|\leq m} \|D^{\beta}f\|_{h^{p}}$ for $f \in C_{c}^{\infty}(\R^{N})$.  The Lemma \ref{lemanovo} is devoted to state the inequality \eqref{5.2a} and its proof follows from reduction of uniform control for bounded atoms and a special smooth atomic decomposition on localizable Hardy spaces when ${N}/{(N+m)}<p<\min\left\{1,N/m\right\}$. The upper bound $N/m$ appears naturally due to Sobolev exponent $p^*:=Np/(N-mp)$ that is crucial for the local comparison between the quasi-norm associated with the elliptic operator $\|A(\cdot,D)f\|_{h^{p}}$ and $\|D^{m}f\|_{h^{p}}$ for test functions, {where $D^{m}f:=\left(D^{\beta}f\right)_{|\beta|=m}$ is the total derivative operator with order $m$.}  

The organization of the paper is as follows. 
In Section \ref{S2} we recall some basic definitions related to localizable Hardy and Hardy-Sobolev spaces. 
The Section \ref{S3} is devoted to a generalization of the Poincar\'e inequality, following the work due to Miyachi \cite{Mi,Mi2}. In Section \ref{S4} we introduce a new smooth atomic decomposition for Hardy-Sobolev spaces, which is a key tool in the proof of inequality \eqref{5.2a}. Sections \ref{S5} and \ref{S6} contain elliptic estimates in Hardy norms and the proof of Theorem A, respectively.
Finally, Section \ref{S7} is devoted to a nonhomogeneous version of the main result and applications to first-order operators.

\bigskip 
\noindent \textbf{Notations}.
Throughout the paper, we will use the notation $\Omega \subset \R^{N}$ for an open set, and by $B_x^t$ denotes a ball $B(x,t)$ centered at $x$ and radius $t>0$ ($B$ denotes a generic ball). We set $S(\R^{N})$ the Schwartz space and $S'(\R^{N})$ the set of tempered distributions. The notation 
$h^p_c(B)$ and $h^{k,p}_c(B)$ will stand for $h^p(\R^N)\cap\E'(B)$ and $h^{k,p}(\R^N)\cap\E'(B)$ respectively, where $\E'(B)$ denotes the distributions with compact support contained in $B$.
We will also make use of the theory of pseudo-differential operators $OpS^{m}(\Omega)$ (see \cite{H}) with symbols in H\"ormander's classes $S^m_{1,0}(\Omega)$. Another basic notation is the Hardy-Little\-wood maximal operator, defined for functions $f \in L^1_{{\rm loc}}(\R^N)$ by
$$
Mf(x) := \; \sup_{x\in B} \fint_B |f(y)|dy, \quad {\rm a.e.} \ x\in\R^N,
$$
where the supremum is taken over all balls that contain $x$. Here $\displaystyle{\fint_B:=\frac{1}{|B|}\int_{B}}$, where $|B|$ is the Lebesgue measure of $B$.  It is well known that $M:f \mapsto Mf$ is a bounded operator in $L^p(\R^N)$ for $1<p\le \infty$, and for $f\in L^\infty(\R^N)$ we have the trivial estimate
$Mf(x) \le \|f\|_{\infty}$, almost everywhere in $x\in\R^N.$ Let $E$ be a finite-dimensional complex vector space of  dimension $n_E$. We say that $f = (f_1,...,f_{n_E}): \R^N \mapsto E$ belongs to a vector spaces $X(\R^N; E)$, if each coordinate function $f_i$ belongs to $X(\R^N)$. 
We denote by $\P_{k}^E$ the space of polynomials $P: \R^N \mapsto E$ of degree at most $k$. In particular, when  $E=\C$, we write simply $\P_{k}$.

\section{Localizable Hardy-Sobolev spaces}\label{S2}

In this section, we recall basic definitions and properties of localizable Hardy spaces and Hardy-Sobolev spaces that will be used throughout the paper. We begin by recalling the definition of the localizable (or local) Hardy space 
$h^p(\erre^N)$, introduced by Goldberg \cite{G}.  Fix, once for all, a radial nonnegative function $\varphi\in\ccinf(\erre^N)$ supported in the unit ball $B(0,1)$ with $\int \varphi =1$.
For a tempered distribution $u\in \S'(\erre^N)$, we define the {\it{local (or truncated) maximal function}}  by
\[ 
m_{\varphi}u(x):=\sup_{0<t<1}|(u*\varphi_t)(x)|,
\] 
where $\varphi_t(x)=t^{-N}\varphi(x/t)$.

\begin{definition} \label{dA.1} 
Let $0<p<\infty$. The localizable Hardy space $h^p(\erre^N)$
is defined as the set of tempered distributions 
$u\in \S'(\erre^N)$ such that  $m_{\varphi}u\in L^p(\erre^N)$, i.e.
$
\|u\|_{h^p} := \|m_{\varphi}u\|_{L^p}<\infty. 
$ 
In the particular case $p=\infty$, we set $h^\infty(\erre^N):=L^\infty(\erre^N)$. 
\end{definition} 

In a more general way, the spaces $h^p(\R^N)$ can be defined using any $\varphi \in S(\R^N)$ with $\int_{\R^N} \varphi \ne0$ and the spaces are independent of the choice of $\varphi$.
For $0<p\le1$, the space $h^p(\erre^N)$ is a complete metric space with the distance 
$
d(u,v)=\|u-v\|_{h^p}^p 
$
for $u,v\in  h^p(\erre^{N})$. When $p=1$, the functional $\|u\|_{h^1}$ is a norm and $h^1(\erre^N)$ is a normed space densely contained in $L^1(\erre^N)$.
When $p>1$, the space $h^p(\erre^N)$ coincides with $L^p(\erre^N)$ and the norms are comparable. 
Although  $h^p(\erre^N)$ is not locally convex for $0<p<1$ 
and $\| \cdot \|_{h^p}$ is truly a quasi-norm (see \cite{Tri1}), we will still refer to $\| \cdot \|_{h^p}$ as a norm by simplicity. We point out that $\|\cdot\|_{h^p}$ is translation invariant.

\subsection{Atomic decomposition}\label{atoms}  
A bounded atom in  $h^p(\erre^N)$ in the sense of Goldberg (see \cite{G,St}) is a measurable function $a$ supported in some ball $B=B(x_{0},r)$ satisfying the following properties:
\begin{enumerate}
\item[(i)] $\|a\|_{L^{\infty}} \leq |B|^{-1/p}$; 
\item[(ii)] if $r<1$ then $\int  a(x)x^\alpha\,dx=0$, for all $\alpha\in\ene^{n}$ satisfying $|\alpha|\le N_{p}\, :=\, \lfloor N(p^{-1}-1) \rfloor $.
\end{enumerate}
These types of bounded atoms are usually called standard $h^{p}-$atoms or simply $h^{p}-$atoms.  
In contrast to atoms in Hardy spaces $H^{p}(\R^N)$, the moment conditions are only required for atoms supported on balls with radius $r<1$. Atoms supported on balls with $r \geq 1$ will be called rough atoms and atoms with radius $r<1$ will be called standard atoms (see \cite{AP} for the terminology). For a deep discussion on moment conditions of atoms in $h^{p}(\R^{N})$ we refer to \cite{D1, D2}. The atomic decomposition theorem for localizable Hardy spaces states that any $f\in h^p(\R^{N})$ can be written as an infinite linear combination of $h^{p}-$atoms, more precisely, there exist a sequence of scalars $\left\{\lambda_j\right\}_{j} \in \ell^{p}(\C)$  and  a sequence of  $h^{p}-$atoms $\left\{a_j\right\}_{j}$ such that the series $\sum_j\lambda_j a_j$
converges to $f$ in  $h^p(\R^{N})$, and consequently in $\S'(\R^{N})$. Furthermore,
$\|f\|_{h^p}^p$ is equivalent to $\inf \sum_j|\lambda_j|^p$,
where the infimum is taken over all atomic representations.  Another
useful fact is that  the atoms may be assumed to be smooth functions. In particular, the inclusions $\ccinf(\erre^N)\subset \S(\erre^N)\subset h^p(\erre^N)$ are dense.
The atomic decomposition of $h^p(\erre^N)$ is quite similar to the
atomic decomposition of $H^p(\erre^N)$ in terms of $H^{p}-$atoms (\cite{St}). The difference is that the notion of an $h^{p}-$atom is less restrictive than that of an 
 $H^{p}-$atom, since an $H^{p}-$atom must satisfy (i) and a stronger form of (ii): its moments are required to vanish regardless of the size of the support.
\begin{remark}
Let $T:\ccinf(\erre^N)\mapsto \D'(\erre^N)$ be a weakly continuous linear operator 
in the sense that $\<\phi_j,\psi\>\to\<\phi,\psi\>$ for all  $ \psi\in\ccinf(\erre^N)$ implies  
$\<T\phi_j,\psi\>\to\<T\phi,\psi\>$ for all $ \psi\in\ccinf(\erre^N)$. 
Given $0<p\le1$, $p\le q\le\infty$, assume that for any smooth $h^{p}-$atom $a$ we have $\|Ta\|_{h^{q}}\le C$, for some uniform constant $C>0$. Then $T$ can be extended to a bounded operator from 
$h^{p}(\erre^N)$ to $h^{q}(\erre^N)$. The proof can be found in \cite{HHP}.
\end{remark}

\begin{definition}
Let $0<r<1$. A continuous function $f$ belongs to the homogeneous H\"older (or Lipschitz) space
$\dot{\Lambda}^{r}(\R^{N})$ if there exists $C>0$ such that, for every $x,h \in \R^{N}$ 
\[
|f(x+h)-f(x)|\leq C|h|^{r}.
\] 
For $r=1$,  $f\in \dot{\Lambda}^1(\R^{N})$ (also called Zygmund space)
if there exists $C>0$ such that for every $x,h\in \R^{N}$
\[
|f(x+h)+f(x-h)-2f(x)|\le C|h| .
\] 
If $r=k+s$ for $k \in \N$ and $0< s \leq 1$, we say that  $f\in \dot{\Lambda}^{r}(\R^{N})$ if all derivatives ${\partial}^\alpha f\in \dot{\Lambda}^{s}(\R^{N})$ for $|\alpha| = k$.
\end{definition}
The homogeneous H\"older space
$\dot{\Lambda}^{r}(\R^{N})$ is a locally convex topological vector space with the seminorm
\[
 [f]_{k+s}\,:=  \sum_{|\alpha| = k}
\sup_{\begin{smallmatrix}x,h \in \R^{N}\\ h \ne 0\end{smallmatrix}} 
\frac{|{\partial}^\alpha f(x+h)-{\partial}^\alpha f(x)|}{|h|^{s}}, \quad \quad  0 < s <1,
\]
or
\[
 [f]_{k+s}\,:=  \sum_{|\alpha| = k}
\sup_{\begin{smallmatrix}x,h \in \R^{N}\\ h \ne 0\end{smallmatrix}} 
\frac{|{\partial}^\alpha f(x+h)+|{\partial}^\alpha f(x-h)-2{\partial}^\alpha f(x)|}{|h|}, \quad \quad s=1,
\]
modulo the subspace of the functions satisfying $[f]_{r}=0$, which consists of polynomials of degree at most 
k.
When $0<p<1$ the dual space of $h^{p}(\R^{N})$ may be identified with the nonhomogeneous H\"older space 
$\Lambda^{r}(\R^{N}):=(\dot{\Lambda}^{r} \cap L^{\infty})(\R^{N})$ for 
{$r = \gamma_p = N\left(\frac{1}{p}-1\right)$},
equipped with the norm $ \|f\|_{r}= [f]_{r}+\|f\|_{L^\infty}$. 
Note that
$N/(N+1)<p<1$ if and only if {$0<\gamma_p<1$}.
The dual of $h^{1}(\R^{N})$ can be identified with the space $\bmo(\R^{N})$, defined as the space of locally integrable functions $f$ satisfy
$$
\|f\|_{bmo}\;:=\; \sup_{|Q|< 1}\frac{1}{|Q|}\int_{Q}|f-f_{Q}|+ \sup_{|Q|\geq 1}\frac{1}{|Q|}\int_{Q}|f|<\infty.
$$
Here $Q$ is a cube in $\R^{N}$ with sides parallel to the axes and 
$$
f_{Q}\; := \; \frac{1}{|Q|}\int_{Q}f(x)dx \;= \fint_{Q}f(x)dx,
$$ 
where $|Q|$ is the Lebesgue measure of Q.

\subsection{Hardy-Sobolev spaces} Let {$k \in \N$}, $0<p<\infty$, and consider on $\ccinf(\erre^N)$ the functional 
\[
\|f\|_{{\dot{h}^{k,p}}}{= \|D^{k}f\|_{h^{p}}}:=  
\sum_{|\alpha|=k}\|{\partial}^{\alpha}f\|_{h^{p}}. 
\]
The completion of $\,\ccinf(\erre^N)$ for the functional $\|f\|_{{\dot{h}^{k,p}}}$ may be identified with a subspace of $\S'(\R^{N})$ denoted by ${\dot{h}^{k,p}}(\R^{N})$ and called the (homogeneous) Hardy-Sobolev space of order k. We claim  if 
$f\in \dot{h}^{k,p}(\R^{N}$) and $k<N$, then ${\partial}^{\alpha} f \in \S'(\erre^N)$ {for $|\alpha|=k$} and 
${\sum_{|\alpha|=k}\|{\partial}^{\alpha}f\|_{h^{p}}<\infty}$. 
 Indeed, suppose $\left\{f_j\right\}_{j}$ is a Cauchy sequence of test functions with respect to the functional ${\|f_j\|_{\dot{h}^{k,p}}=\sum_{|\alpha|=k}\|{\partial}^{\alpha} f_j\|_{h^p}}$ and fix $\psi\in\S(\erre^N)$.
Consider the decomposition 
$$\psi=\sum_{|\alpha|=k} K_\alpha \ast \partial^{\alpha}\psi$$ for $K_\alpha(x)=C_{\alpha,N} x^{\alpha}/|x|^{N}$ (see \cite[Theorem 4.1]{SW1}). 
Thus
\begin{equation}\label{K}
\<f_j,\psi\>=(-1)^{k}\sum_{|\alpha|=k}\<\partial^{\alpha} f_j,K_\alpha*\psi\>
\end{equation}
and {$K_\alpha*\psi$ is a smooth bounded function with bounded derivatives of all orders.
In particular, $K_\alpha*\psi\in \Lambda^r(\erre^N)$ for any $r>0$.} If $0<p<1$ and $\gamma_{p}\, :=\,N(p^{-1}-1)$, then by duality 
\[
|\<f_j-f_k,\psi\>|\le \sum_{|\alpha|=k}|\<{\partial}^{\alpha} (f_j-f_k),K_\alpha*\psi\>|     \le \|f_j-f_k\|_{h^{k,p}} \sum_{|\alpha|=k} \|K_\alpha*\psi\|_{\Lambda^{\gamma_{p}}}
\]
showing that $\left\{f_j\right\}_{j}$ is a Cauchy sequence in $\S'(\erre^N)$. 
In summary, $f\in\S'(\erre^N)$ belongs to ${\dot{h}^{k,p}(\erre^N)}$ if there exists a sequence $f_j\in\ccinf(\erre^N)$ such that
$\left\{f_j\right\}_j$ is Cauchy with respect to the norm 
$g\mapsto\|D^{k} g\|_{h^p}$ and 
$f_j\to f$ in $\S'(\erre^N)$, as we have shown.
It follows that if $f\in {\dot{h}^{k,p}}(\erre^N)$ then ${\partial}^{\alpha} f\in h^p(\erre^N)$ for $|\alpha|=k$ and we set $\|f\|_{{\dot{h}^{k,p}}} := \sum_{|\alpha|=k}\|{\partial}^{\alpha} f\|_{h^p}$.  When $p=1$ the proof is similar by considering $\bmo(\erre^N)$ instead of $\Lambda^r(\erre^N)$. 
{We point out that this topology is Hausdorff, which means it is sufficient to show that if a polynomial $\theta \in \P_{k-1}$ (the class of polynomials with order less or equal than $k-1$) belongs to ${\dot{h}^{k,p}}(\erre^N)$, then $\theta=0$. In fact, let $\left\{f_j\right\}_j \in \ccinf(\erre^N)$ be a Cauchy sequence with norm $g\mapsto\|D^{k} g\|_{h^p}$ that converges to $\theta$ in $\S'(\R^N)$. Thus, $\left\{D^k f_j\right\}_j$ is a Cauchy sequence in $h^p(\R^N)$ and, in particular, $\|D^k f_j\|_{h^p}\to \|D^k \theta\|_{h^p} = 0$. Now consider the alternate sequence $(f_1,0,f_2,0,\dots)$. This is a Cauchy sequence of test functions with respect to the norm $g\mapsto\|D^k g\|_{h^p}$ that converges in $\S'(\R^N)$ to $0$, yet has a subsequence converging to $\theta$. Thus, $\theta=0$.}

Alternatively, we may present another natural definition of Hardy-Sobolev spaces (see \cite{KS} for the case $k=1$). Consider the space of tempered distributions $f\in\S'(\R^N)$ such that $D^{k} f \in h^p(\erre^N)$, equipped with the  semi-norm $\|D^{k} f\|_{h^p}$, and denote it by ${\dot{\textrm{h}}}{}^{k,p}(\erre^N)$. In order to turn it into a Hausdorff space, we must take the quotient 
${\dot{\textrm{h}}}{}^{k,p}(\erre^N)/\{\mathcal{P}_{k-1}\}$. 
Moreover, we have shown that there is exactly one element of ${\dot{h}}^{k,p}(\erre^N)$ in each equivalence class in 
${\dot{\textrm{h}}}{}^{k,p}(\erre^N)/\{\mathcal{P}_{k-1}\}$ so the inclusion 
${\dot{h}}^{k,p}(\erre^N)\subset \dot{\textrm{h}}{}^{k,p}(\R^N)$  induces a natural isometry
${\dot{h}}^{k,p}(\erre^N)\simeq {\dot{\textrm{h}}}{}^{k,p}(\erre^N)/\{\mathcal{P}_{k-1}\}$. 
{We denote by $h^{k,p}(\R^{N})$ the nonhomogeneous  Hardy-Sobolev 
space, defined as the completion of $\,\ccinf(\erre^N)$ with respect to the quasi-norm $
\|f\|_{h^{k,p}}:=  
\sum_{|\alpha|\leq k}\|{\partial}^{\alpha}f\|_{h^{p}}.$}

\section{Poincar\'e inequality}\label{S3}

Let $\Omega$ be an open set in $\R^{N}$ and fix $\phi \in C_c^{\infty}(B(0,1))$ with $\int \phi =1$. For each $0 < a \leq \infty$ and a distribution $f \in \mathcal{D}'(\Omega)$, we define the maximal functions	
$$f^+_{\phi, \Omega}(x) = \displaystyle \sup_{0<t<d(x, \partial \Omega)} |\inner{f}{\phi_t(x-\cdot )}| \quad \text{and} \quad f^{+, a}_{\phi, \Omega}(x) = \displaystyle \sup_{0<t<\min\{a,d(x, \partial \Omega)\}} |\inner{f}{\phi_t(x-\cdot )}|$$
for each $x \in \Omega$. Furthermore, for each $x \in \R^N, t>0$ and $L \in \N$, we define
$$\mathcal{F}_{L,\Omega}(x,t) = \{ \phi \in C_c^{\infty}(B(x,t)\cap \Omega), \norma{\partial^{\alpha} \phi}_{\infty} \leq t^{-N-|\alpha|}, |\alpha| \leq L\},$$
and the function
\begin{equation}
f^{*}_{L,\Omega,a}(x) = \displaystyle \sup_{ \displaystyle {\phi \in \bigcup_{0<t<a}} \mathcal{F}_{L,\Omega}(x,t)} |\inner{f}{\phi}|.
\end{equation}
When $\Omega=\R^{N}$, we simplify the notation $f^{*}_{L,\Omega,a}$ by $f^{*}_{L,a}$. The following definition was presented by Miyachi in \cite{Mi}.

\begin{definition} 
Let $0<p<\infty$ and $f \in \D'(\Omega)$. We say that  $f$ belongs to the Hardy space $H^p(\Omega)$ or to the nonhomogeneous Hardy space $h^p(\Omega)$ if $f^+_{\phi, \Omega} \in L^p(\Omega)$ or if $f^{+, 1}_{\phi, \Omega} \in L^p(\Omega)$, respectively.  Furthermore, {if $0<p \leq 1$,} we say that $f \in H^p_{loc}(\Omega)$ if, for each $x \in \Omega$, there exists an open neighborhood $U_x \subset \Omega$ such that $f|_{U_x} \in H^p(U_x)$. 
\end{definition}
Clearly, for each $x \in \Omega$, we have $f^{+, a}_{\phi, \Omega}(x) \leq C_{\phi, L} f^{*}_{L,\Omega,a}(x)$; by {a non trivial argument 
	\begin{equation}\label{eq_*_+}
		\|f^{*}_{L,\Omega,a}\|_{L^{\gamma}(\Omega)}\leq C_{\phi,L,a}\|f^{+, 1}_{\phi, \Omega} \|_{L^{\gamma}(\Omega)}
	\end{equation}
	for $0<\gamma<\infty$, $L>N/\gamma-N$ and $a>0$.} Roughly speaking,  \cite[Corollary 2]{Mi2} asserts the pointwise control $f^{*}_{L,\Omega,a}(x) \leq C_{\phi, L} M_{\frac{N}{N+L}}\left[f^{+, a}_{\phi, \Omega}\right](x)$ for any $x \in \Omega$ and $a,L>0$, where the operator $M_{q}$ is the fractional maximal operator given by $M_q[g](x):= \displaystyle \sup_{t>0}\left( \dfrac{1}{t^N} \int_{B_{x}^{t} \cap \Omega} |g|^q \right)^{{1}/{q}}$, which is bounded on $L^p(\Omega)$ for all $p>q$. 
In particular, this shows that $H^{p}(\Omega)$ and $h^{p}(\Omega)$ do not depend on the choice of $\phi$. When $\Omega=\R^{N}$, we have $f^{+}_{\phi, \R^N} = M_{\phi} f$ and $f^{+,1}_{\phi, \R^N} = m_{\phi} f$; thus we recover the classical Hardy spaces $H^{p}(\R^{N})$ and $h^{p}(\R^{N})$, respectively.

Given $f \in \D'(\Omega)$ and $x \in \Omega$, if the limit $\displaystyle \lim_{t\searrow0} f\ast{\phi_t(x)}$ exists and does not depend on the choice of $\phi$, we define 
		$$[f](x) = \lim_{t\searrow0}  f\ast{\phi_t(x)}.$$
We point out that if $ f \in L^1_{loc}(\R^N)$ then $[f]$ exists and $[f] = f$ almost everywhere. The following construction is due to Miyachi. 
	\begin{lemma}\label{Lema4_M} \cite[Lemma 4]{Mi}
		Let $m \in \N$. Then for each cube $Q \subset \R^N$ there exists a function $\phi_Q$ on $\R^N\times \R^N$ with the following properties:
		\begin{enumerate}
			\item [(i)] $\phi_Q \in C^{\infty}(\R^N \times \R^N)$, and $\supp \phi_Q(x, \cdot) \subset \mathring{Q}$ for each $x \in \R^N$, where $\mathring{Q}$ denotes the interior of $Q$.
			\item [(ii)] For each $f \in \D'(\mathring{Q})$, the function $P_{Q,f}: \R^N \mapsto \C$ given by $P_{Q,f}(x) = \inner{ f}{\phi_Q(x,\cdot)}$ belongs to $\P_{m-1}$; i.e. it is a polynomial of degree at most $m-1$.
			\item [(iii)] $P(x) = \int P(y) \phi_Q(x,y) dy$, for all $P \in \P_{m-1}$.
			\item [(iv)] Let $a>0$ and $x \in aQ$. Then 
$$|\partial^{\alpha}_x \partial^{\beta}_y \phi_Q(x,y)| \leq C_{m,a,\alpha,\beta} \ell(Q)^{-N-|\alpha|-|\beta|},$$ for all multi-indexes $\alpha, \beta$.
		\end{enumerate}
	\end{lemma}
Let $Q_{0}$ be the unit cube centered at the origin in $\R^N$ and fix $\eta \in C^{\infty}_c(Q_{0})$ such that $\eta \geq 0$ and $\int \eta = 1$. Consider  $\P_{m-1}$ the Hilbert space of the polynomials of degree at most $m-1$ with inner product 
$$\inner{p}{q}_{M} = \int p(x)\bar{q}(x)\eta(x)dx,$$ 
and $\{\pi_i\}_i$ an orthonormal basis of $(\P_{m-1}, \inner{\cdot}{\cdot}_{M}) $. For each cube $Q \subset  \R^N$, we define
		\begin{equation}\label{def_phi_Q}
			\phi_Q (x,y) = \ell(Q)^{-N} \left[ \sum_{i=1}^m \pi_i\left(\dfrac{x-x_Q}{\ell(Q)} \right) \overline{\pi_i\left(\dfrac{y-x_Q}{\ell(Q)} \right)} \; \right] \eta\left(\dfrac{y-x_Q}{\ell(Q)} \right),
		\end{equation}
which satisfies all the desired properties.

An important consequence of this lemma is the following result: 
\begin{proposition}\label{prop_3.2}
Let $0< p \leq 1$. If $f \in \D'(\Omega)$ satisfies $\partial^{\alpha}f \in H^p_{loc}(\Omega)$ for all $|\alpha|=m$, then for each $Q \subset \Omega$ and $L \in \N$, there exist a polynomial $P_{Q,f} \in \P_{m-1}$ and {a positive constant $C=C(m,p,L)$} such that
\begin{equation}\label{eq_M}
			|[f](x) - P_{Q,f}(x)| \leq C \ell(Q)^m \sum_{|\alpha|=m} (\partial^{\alpha}f)_{L,\Omega,a}^{*}(x), \quad a=\sqrt{N}\ell(Q),
\end{equation}
for $x \in Q$ almost everywhere.
\end{proposition}

The proof can be found in \cite[pp.86 inequality (5.15)]{Mi}. We are interested in using this result when $\Omega=\R^N$ and for functions in $(L^{1}_{loc} \cap h^{p})(\R^N)$. In order to do it, we present the following definition in \cite{Mi}.

\begin{definition} 
	Let $\Omega \subset \R^N$ be an open subset. We say that $\Omega$ satisfies the condition $(\star)$ if for some constant $A>1$ and each $x \in \Omega$ with $d(x,\Omega^{c})< A^{-1}$, there exists  $x' \in \Omega^c$ such that 
	$$		d(x,x') < A \; d(x, \Omega^{c})
		\quad \text{and} \quad 
		d(x',\Omega) > A^{-1} \; d(x, \Omega^{c}).
	$$
\end{definition}
In particular, any open ball on $\R^N$ satisfies the condition $(\star)$.  
Moreover, from \cite[Remarks (b) and (c), pp. 77]{Mi}, we have $\displaystyle \left. h^p(\R^N)\right|_{_{\Omega}} = h^p(\Omega)$ if $\Omega$ satisfies the condition $(\star)$ and $h^p(\Omega) = H^p(\Omega)$ if $\Omega$ is bounded. Now, let $f \in h^p(\R^N)$ and for each $x$ in $\R^N$ denote $U_x := B(x,1)$. Then,
	$\displaystyle \left. f \right|_{_{U_x}} \in \left. h^p(\R^N)\right|_{_{U_x}} = h^p(U_x) = H^p(U_x)$
	and by definition $f \in H^p_{loc}(\R^N)$.

\subsection{Generalization of Poincar\'e Inequality when $\Omega=\R^{N}$}
Let $0< p \leq 1$ and $f \in L^1_{loc}(\R^N)$ such that $\partial^{\beta}f \in h^p(\R^{N})$ for all $|\beta|=m$. For each $B(x,t) \subset \R^N$ with $0<t<T$ for some fixed $T>0$, there exists a polynomial $P_{x,t,f} \in \P_{m-1}$ and {a constant {$C_{m,N,T}>0$} such that
\begin{equation}\label{eq_M_adapt}
|f(y) - P_{x,t,f}(y)| \leq C_{m,N,T} \; t^m \sum_{|\beta|=m} (\partial^{\beta}f)_{N_p,a}^{*}(y), \quad a=2\sqrt{N}t,
\end{equation}
for $y \in B(x,t)$ almost everywhere. In fact, for each open ball $B(x,t) \subset \R^N$ and $f \in L^1_{loc}(\R^N)$ we can consider the open cube $Q:={Q}(x,2t)$ centered at $x$ and side $\ell(Q)=2t$, and define $\tilde{f} = f \mathcal{X}_{B(x,t)}$.} Since $\tilde{f} \in \D'(Q)$, Lemma \ref{Lema4_M} ensures that there exits a test function $\phi_Q$ such that $P_{Q, \tilde{f}}(y) = \inner{\tilde{f}}{\phi_Q(y,\cdot)}$ is a polynomial of degree at most $m-1$. In particular, 
	$$P_{Q, \tilde{f}}(y) = \int_Q \tilde{f}(z) \phi_Q(y,z) dz = \int_{B_x^t} f(z) \phi_Q(y,z) dz.$$
	For simplicity, we define $P_{x,t,f} \, := \, P_{Q(x,2t),\tilde{f}}$ and we can apply the Proposition \ref{prop_3.2} {for $\tilde{f} \in \D'(\R^N)$}, noting that $[\tilde{f}]=f$ in $B(x,t)$ almost everywhere. Integrating the inequality \eqref{eq_M_adapt} over $B(x,t)$, we get
	\begin{eqnarray*}
		\left( \fint_{B(x,t)} \left|\frac{1}{t^m}\left[ f(y) -  P_{x,t,f}(y)\right]\right|^{\alpha} dy \right) ^{{1}/{\alpha}} &\leq& \dfrac{C}{|B_x^t|^{\frac{1}{\alpha}}} \norma{ \sum_{|\beta|=m} (\partial^{\beta}\tilde{f})_{N_p, 2\sqrt{N}t}^{*} }_{L^\alpha(B_x^t)}\\
		&\leq& \dfrac{C}{|B_x^t|^{\frac{1}{\gamma}}} \norma{ \sum_{|\beta|=m} (\partial^{\beta}\tilde{f})_{N_p, 2\sqrt{N}t}^{*} }_{L^\gamma(B_x^t)},
	\end{eqnarray*}
satisfying $\gamma = \alpha$ if $\alpha>1$ and  $\gamma>\alpha$ if $\alpha=1$. By inequality \eqref{eq_*_+}, 
	\begin{align*}
		\norma{ \sum_{|\beta|=m} (\partial^{\beta}\tilde{f})_{N_p, 2\sqrt{N}T}^{*} }_{L^\gamma(B_x^t)} 
		\leq C_{\phi, N,p, T}\sum_{|\beta|=m} \norma{(\partial^{\beta} \tilde{f})^{+, 1}_{\phi, \R^N}}_{L^\gamma(\R^N)} 
		&\lesssim \sum_{|\beta|=m} \norma{\partial^{\beta} \tilde{f}}_{L^\gamma(\R^N)} \\
		&\lesssim \norma{D^{m} f}_{L^\gamma(B_x^t)}.
	\end{align*}
	Summarizing, there exists a constant $C=C(N,p, m, T, \alpha)>0$ such that
	\begin{equation}\label{Desigualdade_Poincarè_M_maior}
		\left( \fint_{B_x^t} \left|\frac{1}{t^m} [f(y) -  P_{x,t,f}(y)]\right|^{\alpha} dy \right) ^{{1}/{\alpha}} \leq C \left( \fint_{B_x^t} | D^m f(y)|^{\gamma} dy \right) ^{\frac{1}{\gamma}},
	\end{equation}
where $\gamma=\alpha>1$ or $\gamma>\alpha=1$.


\section{New smooth atomic decomposition on Hardy-Sobolev spaces}\label{S4}

Let us first recall the standard atomic decomposition on $h^p(\erre^N)$ due to Goldberg, which uses the maximal characterization (\cite[Lemma 4]{G}).
Let $\varphi\in \ccinf(\erre^N)$ with $\int \varphi(x)\,dx=1$ and $\int x^\alpha\varphi(x)\,dx=0$ for $0<|\alpha|\le L(p)$ sufficiently large (bigger than $N_{p}:= \lfloor N \left( \frac{1}{p}-1\right) \rfloor$) and write  $f_1 := f-\varphi*f$, $f_2 \, :=\,\varphi*f$ that satisfy 
\begin{equation}\label{a1}
f=f_1+f_2, \,\quad \|f_1\|_{H^p}\le C\|f\|_{h^p},\quad	
\forall \, f\in h^p(\erre^N),
\end{equation}
with $C>0$ independent of $f$. 
Thus, an atomic decomposition of $f$ is obtained by choosing an $H^p$ atomic decomposition 
for $f_1$ and an $h^p$ atomic decomposition for $f_2$, which we discuss next.
Choose a function $\psi \in\ccinf(\erre^N)$ supported in the cube $Q_{0}^{2}$ centered at the origin with side length 2, such that $0\le \psi(x)\le 1$ and $\displaystyle{\sum_{k\in\ze^N}\psi(x-k)=1}$ for all $x \in \R^N$.
Setting $\psi_k(x) := \psi(x-k) \in C_{c}^{\infty}(Q_{k}^{2})$, we may define $\mu_k := |Q^2_k|^{1/p}\|\psi_kf_2\|_{L^{\infty}(Q^2_k)}$ and $a_{k}\,:=\,\psi_kf_2/\mu_k$ whenever $\psi_kf_2\ne0$, obtaining the  decomposition 
\begin{equation}\label{f2}
f_2=\sum_k \psi_k f_2= \sum_k \mu_{k}a_k.
\end{equation}
Clearly $\psi_k f_2\in\ccinf (Q^2_k)$; thus $a_{k}$ is a rough atom, since $\|a_{k}\|_{L^{\infty}}\leq |Q^{2}_{k}|^{-1/p}$. We claim that
 \begin{equation}\label{desig}
\sum _k|\mu_k|^p\le C\|f\|_{h^p}^p 
\end{equation}
 To show this control, we recall that given $R>0$ there exists $C>0$ such that
\begin{equation*}
|\varphi*f(x)| \le C\, \mathfrak{m}f(y),\quad \quad |x-y|\le R,
\end{equation*}
where $\mathfrak{m}$ denotes the \textit{grand maximal function} used to define an equivalent norm in $h^p$, i.e. $\|f\|_{h^p} \simeq \|\mathfrak{m}{f}\|_{L^p}$ (see \cite[Theorem 1]{G}). Then 
\[
\sup_{Q_k^2}|\varphi*f(x)|^p\le C\frac{1}{|Q_k^2|}\int_{Q_k^2}[{\mathfrak{m}}f(y)]^p\,dy
\]
which implies that 
$\sum_k|Q_k^2|\sup_{Q_k^2}|\varphi*f(x)|^p \le C\|\mathfrak{m}f\|^p_{L^p}\simeq\|f\|^p_{h^p}$, 
because the family of cubes $\{Q^2_k\}$ has the bounded intersection property. This concludes \eqref{desig}.   

In the previous proof, we could simplify taking $\psi\,:=\, \chi_{B(0,1)}$. In this case, the atoms $a_{k}$ in \eqref{f2} are not smooth. The advantage of using a cut-off function is that for $f\in C_{c}^{\infty}(\R^N)$, thus  $f_{1},f_{2}\in C_{c}^{\infty}(\R^N)$,  we may conclude the rough atoms $a_{k}=\psi_{k}f_{2}/\mu_{k} \in C_{c}^{\infty}(Q^{2}_{k})$ and the series  $\sum_k \mu_k a_k$ reduces to a finite sum, since there is finite indices such that $Q^{2}_{k} \cap supp(f) \neq  \varnothing$. {In this particular setting, we may also obtain a finite atomic decomposition for bounded atoms in $H^{p}$ for $f_{1}$ (see \cite[Theorem 3.1 and Remarks 3.2 and 3.3]{MSV})}. Invoking the atomic decomposition for $H^p$, we may write
\begin{equation}\label{sum}
f= \sum_k \mu_{1,k}a_{1,k}+ \sum_k\mu_{2,k}a_{2,k}
\end{equation} 
where $a_{1,k}$ are $H^p$ atoms and $a_{2,k}$ are rough atoms, and 
$\sum _k \left(|\mu_{1,k}|^p + |\mu_{2,k}|^p \right)\le C\|f\|_{h^p}^p.$
From the previous discussion, we may also obtain an atomic decomposition where the sum \eqref{sum} is finite, assuming $f \in \left(h^{p}\cap C_{c}^{\infty}\right)(\R^{N})$.

\begin{lemma} \label{remarkcatarina}
If $f \in \S(\R^N)$, then clearly $f_2 \,:= \,\varphi * f \in \S(\R^N)$. The partial sums  $\displaystyle{\sum_{|k|\le M}\mu_k a_k}$ converge to $f_2$ in $\S(\erre^N)$, and \textit{a fortiori} in $\Lambda^r(\erre^N)$ for any $r>0$.
\end{lemma}

\begin{proof}
First, we claim $\displaystyle{\sum_{|k|\le j} \psi_k f_2}$ converges to $f_2$ in $\S(\erre^N)$ as $j \to \infty$, where $\psi_{k}$ is given above.  
We point out that, due to compact support of $\psi$, for each $x \in \R^N$, there are at most $2^N$ indexes $k \in \Z^N$ such that $\psi_{k}(x)\,:=\, \psi(x-k) \neq 0$. Thus, the sum $\sum_{k\in\ze^N}\psi(x-k)=1$ is always finite,  and for each non-zero multi-index $\gamma \in \Z^N$  we have
$\sum_{k\in\ze^N}\partial^{\gamma} \psi_{k}(x)=0$. 
The conclusion follows from the identity 
\begin{eqnarray*}
 \partial^{\beta} \left( \sum_{|k|\le j} \psi_k(x)f_2(x) \right) &=& \left( \sum_{|k|\le j} \psi_k(x) \right)\partial^{\beta}f_2(x)
+\sum_{0<\gamma \leq \beta} \binom{\beta}{\gamma}  \left( \sum_{|k|\le j} \partial^{\gamma}\psi_k(x) \right) \partial^{\beta-\gamma}f_2(x).
\end{eqnarray*}
The second claim follows by canonical embedding, for instance if $0<r<1$, then
\begin{align*}
	[f]_r 
	&\leq \sup_{\substack{x \in \R^N\\ |h|<\delta}} \dfrac{|f(x+h) - f(x)|}{|h|^r} + \sup_{\substack{x \in \R^N\\ |h|\geq \delta}} \dfrac{|f(x+h) - f(x)|}{|h|^r}\\
	&\leq \|\nabla f \|_{L^\infty}. \sup_{\substack{x \in \R^N\\ |h|<\delta }} |h|^{1-r} + \delta^{-r} \sup_{\substack{x \in \R^N\\ |h|\geq \delta}} |f(x+h)| + |f(x)|\\
	&\leq \delta^{1-r} \|\nabla f \|_{L^\infty} + \delta^{-r}\|f \|_{L^\infty}.
	\end{align*}
The general case can be easily extended. \qed
\end{proof}

\subsection{A generalized Calder\'on-Zygmund decomposition}

{Consider $m\in\ene$, $N/(N+m)<p<1$, and denote $\gamma_{p}:=N\left(\frac{1}{p}-1\right)$, $N_{p}:= \lfloor \gamma_{p} \rfloor$ and $r_{p} \; := \; \gamma_{p}-N_p$. Recall that if $N_p < \gamma_p$, then $f \in \Lambda^{\gamma_{p}}(\R^N)$ means ${\partial}^{\alpha}f \in \Lambda^{r_{p}}(\R^N)$ for all $|\alpha| = N_p$; if $N_p = \gamma_p$ then  ${\partial}^{\alpha}f \in \Lambda^{1}(\R^N)$ for all $|\alpha| = \gamma_p -1$. 
{Furthermore, when $p=1$ we consider $\bmo(\erre^N)$ instead of $\Lambda^{\gamma_{p}}(\erre^N)$.}

In this section, we are interested in obtaining the following result:

\begin{proposition}\label{czespecial}
{Let $m,p$ and $\gamma_{p}$ as before. Given $f \in \left(C^{m}_c \cap h^{p}\right)(\erre^N)$ there exists a sequence of functions $f_n\in C^{m}_c(\erre^N)$} with the following properties:
\begin{enumerate}
  \item[(i)] there exists $C>0$ such that $\|f_n\|_{h^p}\le C \|f\|_{h^p}$;
  \item[(ii)] { $\displaystyle\lim_{n\to\infty}\norma{f-f_n}_{\Lambda^{\gamma_{p}}}=0$, when $N/(N+m)<p<1$, and $\displaystyle\lim_{n\to\infty}\norma{f-f_n}_{bmo}=0$, when $p=1$}.
  \end{enumerate}
\end{proposition}

The proof is reduced for the case when {$f\in (H^p\cap C^{m}_c)(\erre^N)$}. In fact, using the previous decomposition  \eqref{a1} we may split $f=f_1+f_2$ with {$f_1 \in (H^p\cap C^{m}_c)(\erre^N)$} and $f_2 \in C_{c}^{\infty}(\R^{N})$. The conclusion follows from Lemma \ref{remarkcatarina}. To simplify the notation, we write $f$ instead of  $f_1$ and assume that {$f\in (H^p\cap C^{m}_c)(\erre^N)$} in order to prove the result. 

Let us first fix some notations in the spirit of Stein`s book \cite[Chapter III, Section 2]{St}. Namely, $\M$ will denote the grand maximal function, and $\M_0=M_\varphi$ is the maximal function associated with the single function $\varphi \in S(\R^{N})$ satisfying $\int \varphi \neq 0 $; i.e., 
$M_\varphi f(x)=\sup_{t>0}|f*\varphi_t(x)|$. The grand maximal function is associated with a collection $\mathfrak{S}\subset\S(\erre^N)$ of rapidly decreasing functions, 
\begin{equation*}
\displaystyle \M f(x)=\sup_{\phi\in\mathfrak{S}}M_\phi f(x).
\end{equation*}
We may assume without loss of generality that $\mathfrak{S}$ is a bounded subset 
of $L^1(\erre^N)$. In particular, if $u\in L^\infty(\erre^N)$ then 
$\M u\in L^\infty(\erre^N)$.

Next we state an extension of the Calder\'on-Zygmund decomposition for smooth functions in Hardy spaces. 

	\begin{proposition}\label{decomp_f_g_b}
		Let $m \in \N$ and $p$ as before. For each {$f\in \left(C^{m}_c\cap H^p\right)(\erre^N)$} and $\alpha>0$, there exist a decomposition $f=g+b$ with $b := \,\sum_{k=1}^\infty \beta_k$, and a collection of cubes $\{Q_k^*\}$ with sides of length $\ell_k$  parallel to the axes, such that
		\begin{enumerate}
			\item [(i)] $|g(x)|\le c\,\alpha$ almost everywhere;
			\item [(ii)] {Each $\beta_k \in C_{c}^{m}(Q^*_k)$} satisfies the following properties:
			{
				\begin{align}
					\|D^{i} \beta_k\|_{L^\infty}&\le c \,\sum_{j=0}^{i} \ell_k^{{m}+j-i}\;  \norma{D^{{m}+j}f}_{L^{\infty}}, \quad \forall \, i \in \{0,1,...,m-1\}, \label{a3_{m}} \\
					\|D^{i} \beta_k\|_{L^\infty}&\le c \,  \norma{D^{i}f}_{L^{\infty}}, \quad \forall \, i \in \{1,2,...,m\}, \label{a4_{m}} \\
					\int_{\R^N}(\M_0 \beta_k)^p\,dx &\le c\int_{Q_k^*}(\M f)^p\, dx, \label{des_integral_{j_p}} \\	
					\int x^{\gamma}\beta_k(x) dx&=0, \,\, \quad   \forall \, |\gamma|  \leq m-1. \label{moment_{j_p}}
			\end{align}}
			\item [(iii)] The cubes $\{Q_k^*\}$ have the bounded intersection property, and 
			$$ \O \, :=\, \{x:\ \M f(x)>\alpha\}=\bigcup_k Q_k^*.$$
		\end{enumerate}
	\end{proposition}
	\begin{proof}
		We define $b=\sum_k\beta_k$ and $g=f-b$ analogously as the generalized Calder\'on-Zygmund decomposition {\cite[pp.101]{St}}, precisely: fix a partition of unity $0\le\eta_k(x)\le1$  for $\O$ subordinated to the covering $\{Q_k^*\}$ satisfying $\supp\eta_k\subset Q_k^*$ and 
		$|\partial^\gamma\eta_k|\le C_{\beta,\gamma}\ell_k^{-|\gamma|}$, 
		and
		\begin{equation}\label{bk_{j_p}}
			\beta_k(x) = \, (f(x)-P_{f,k}(x))\eta_k(x), 
		\end{equation}
		where $\displaystyle{P_{f,k}(x) = \int f(y) \phi_{Q_k} (x,y)dy}$ is the polynomial with degree at most $ m-1$ 
		given by Lemma \ref{Lema4_M}, and 
		\begin{align*}
			\phi_{Q_k} (x,y) &= \ell_k^{-N} \left[ \sum_{i=1}^{{m}} \pi_i\left(\dfrac{x-x_k}{\ell_k} \right) \overline{\pi_i\left(\dfrac{y-x_k}{\ell_k} \right)} \; \right] \eta_k(y) 
= \left[ \sum_{i=1}^{{m}} e_i\left(x \right) \overline{e_i\left(y \right)} \; \right] \tilde{\eta_k}(y).
		\end{align*}
		Here $\tilde{\eta_k}\,  :=\,  {\eta_k}{\left(\int \eta_k dx\right)^{-1}}$ and $e_i(x) \,  := \, \left(\int \eta_k dx\right) \ell_k^{-N/2} \pi_i\left(\frac{x-x_k}{\ell_k} \right)$, then $\{e_1, ..., e_{m}\}$ defines an orthonormal basis of $(\P_{m-1}, \inner{\cdot}{\cdot}_{S})$, where $\inner{p}{q}_{S} = \int p \bar{q}\tilde{\eta_k}$. 
		
		The conclusions of claims (i) and (iii), as well as the inequality \eqref{des_integral_{j_p}}, follow similarly to the arguments presented in \cite{St} and will therefore be omitted. The property \eqref{moment_{j_p}} follows directly from \eqref{bk_{j_p}}, precisely 
		\begin{eqnarray*}
			\int P_{f,k}(x)\eta_k(x)\overline{e_{j_0}}(x) \;dx 
			&=& \int \left( \int f(y)\left[ \sum_{i=1}^m e_i\left(x \right) \overline{e_i}\left(y \right) \right] \tilde{\eta_k}(y) \;dy\right)\eta_k(x)\overline{e_{j_0}}(x) \;dx\\
			&=& \int f(y) \tilde{\eta_k}(y) \left[ \sum_{i=1}^m \overline{e_i} \left(y \right) \left( \int  e_i\left(x \right) \eta_k(x)\overline{e_{j_0}}(x) \; dx\right) \right] dy\\
			&=& \int f(y) \eta_k(y) \left[ \sum_{i=1}^m \overline{e_i} \left(y \right) \left( \int  e_i\left(x \right) \tilde{\eta_k}(x)\overline{e_{j_0}}(x) \; dx \right)  \right] dy\\
			&=& \int f(y) \eta_k(y) \overline{e_{j_0}}(y) \;dy,
		\end{eqnarray*}
		for each ${j_0} \in \left\{1,...,{m}\right\}$, thus $\displaystyle{\int \beta_k(x)q(x) dx = 0}$ for all $q \in \P_{m-1}$. {Since $N_p \leq m-1$ we get the desired property.} 
		
It follows from inequality \eqref{eq_M} and the uniform control $f^{*}_{L,a}(x) \leq C(N) \|f\|_{L^{\infty}}$ for all $f \in L^{\infty}(\R^{N})$ that
		$$|\beta_k(x)| \leq |f-P_{f,k}|(x) \leq C_{{m},p,M} \ell_k^{{m}} \sum_{|\alpha|={m}} (\partial^{\alpha}f)_{L,\sqrt{N}\ell_k}^{*}(x) \lesssim \,\ell_k^{m} \,\|D^{m} f\|_{L^\infty},$$
		proving the inequality \eqref{a3_{m}} for $i=0$. For each fixed $k \in \N$ and multi-index $\alpha$ with $|\alpha|=j \in \{1,2,...,m\}$, let $\displaystyle P^{j-1}_{x_k,f}$ denote the Taylor polynomial of $f$ centered at $x_k$ with degree $j-1$. Then
		$$\left| f(x) - P^{j-1}_{x_k,f}(x) \right| \lesssim |x-x_k|^{j} \sup_{|\gamma|=j} \left|\partial^{\gamma}f(\tilde{x_k})\right|,$$
		for some $\tilde{x_{k}}\in [x,x_k]$. In particular, $\left| f(x) - P^{j-1}_{x_k,f}(x) \right| \lesssim \ell_k^{j} \|D^j f\|_{L^\infty}$ for each $x \in Q_k^*$. It follows from the Leibniz rule that
		\begin{align}\label{equation_AB}
			|\partial^{\alpha} \beta_k(x)|  & \lesssim  
			\sum_{0\leq\gamma < \alpha} \left|  \partial^{\gamma}(f-P_{f,k})(x) \partial^{\alpha-\gamma}\eta_k(x) \right| 
			+
			\sum_{|\gamma|=j} \left|  \partial^{\gamma}f(x)\eta_k(x) \right| \nonumber\\
			& \lesssim 
			\sum_{0\le \gamma < \alpha} \ell_k^{|\gamma|-j}|\partial^{\gamma}(f-P_{f,k})(x)| +\|D^{j}f\|_{L^{\infty}} \nonumber \\
			& \lesssim 
			\sum_{0\le \gamma < \alpha} \ell_k^{|\gamma|-j} \left[ \underbrace{ |\partial^{\gamma}f(x) - \partial^{\gamma} P^{j-1}_{x_k,f}(x)|}_{(A)} + \underbrace{|\partial^{\gamma}(P^{j-1}_{x_k,f} - P_{f,k})(x)|}_{(B)} \right] +\|D^{j}f\|_{L^{\infty}}. 
		\end{align}
		Since $\partial^{\gamma}P^{j-1}_{x_k,f} = P^{j-|\gamma|-1}_{x_k,\partial^{\gamma}f}$, we get
		\begin{equation*}
			(A) = |\partial^{\gamma}f(x) - P^{j-|\gamma|-1}_{x_k,\partial^{\gamma}f}(x)| \lesssim \ell_k^{j-|\gamma|} \| D^{j-|\gamma|} (\partial^{\gamma}f)\|_{L^\infty} \simeq \ell_k^{j-|\gamma|} \| D^{j} f\|_{L^\infty}.
		\end{equation*}
		On the other hand, it follows from (iii) of Lemma \ref{Lema4_M} that 
		$$P^{j-1}_{x_k,f}(x) = \int_{Q_k} P^{j-1}_{x_k,f}(y) \phi_{Q_k}(x,y) \; dy, \text{ for all } x \in Q_k.$$
		This way,
		$$P^{j-1}_{x_k,f}(x) - P_{f,k}(x) = \int_{Q_k} \left[ P^{j-1}_{x_k,f}(y) -f(y) \right]\phi_{Q_k}(x,y) \; dy$$
		and
		\begin{eqnarray*}
			(B) = |\partial^{\gamma}(P^{j-1}_{x_k,f} - P_{f,k})(x)|
			&=& \left| \int_{Q_k} \left( P^{j-1}_{x_k,f}(y) -f(y) \right)\partial^{\gamma}_x \phi_{Q_k}(x,y) \; dy \right| \\
			&\lesssim& \ell_k^j \|D^jf\|_{L^\infty} \ell_k^{-N-|\gamma|}\ell_k^N \\
			&=& \ell_k^{j-|\gamma|}\|D^jf\|_{L^\infty}
		\end{eqnarray*}
		Substituting into \eqref{equation_AB}, we obtain
		\begin{eqnarray*}
			|\partial^{\alpha} \beta_k(x)| &\lesssim& \sum_{0\le \gamma < \alpha} \ell_k^{|\gamma|-j} \left[ \ell_k^{j-|\gamma|} \| D^{j} f\|_{L^\infty} + \ell_k^{j-|\gamma|}\|D^jf\|_{L^\infty} \right] +\|D^{j}f\|_{L^{\infty}},
		\end{eqnarray*}
the we have \eqref{a4_{m}}. Furthermore, for each $\alpha$ such that $|\alpha|=i \in \{1,2,...,{m-1}\}$, by the Leibniz rule we have
		\begin{eqnarray}
 \label{eq_beta_m} \quad \quad  |\partial^{\alpha} \beta_k(x)|  \lesssim 
			\sum_{0\leq\gamma \leq \alpha} \left|  \partial^{\gamma}(f-P_{f,k})(x) \partial^{\alpha-\gamma}\eta_k(x) \right| \lesssim 
			 \sum_{0\le \gamma \le \alpha} \ell_k^{|\gamma|-i}|\partial^{\gamma}f(x) - \partial^{\gamma} P_{f,k}(x)|.
		\end{eqnarray}
Since $\partial^{\gamma} P_{f,k}(x) = P_{\partial^{\gamma}f,k}(x) + (II)$ (see \cite[p. 87]{Mi}), in which
		$$|(II)| \leq C_{{m},M,\gamma} \, \ell_k^{{m}-|\gamma|}  \sum_{|\beta|={m}} (\partial^{\beta}f)_{L,\sqrt{N}\ell_k}^{*}(x) \lesssim \ell_k^{{m}-|\gamma|}\|D^{m} f\|_{L^\infty},$$
		we get
		\begin{eqnarray*}
			|\partial^{\gamma}f(x) - \partial^{\gamma} P_{f,k}(x)| \lesssim \ell_k^{{m}}\;  \sum_{|\theta|={m}} \norma{\partial^{\theta+\gamma}f}_{L^{\infty}} + \ell_k^{{m}-|\gamma|}\|D^{m} f\|_{L^\infty}.
		\end{eqnarray*}
Substituting this into inequality \eqref{eq_beta_m} yields
		\begin{eqnarray*}
			\|D^i \beta_k\|_{L^\infty} \lesssim \sum_{j=0}^{i} \ell_k^{{m}+j-i}\;  \norma{D^{{m}+j}f}_{L^{\infty}},
		\end{eqnarray*}
		concluding the proof with the inequality \eqref{a3_{m}}. \qed
		
	\end{proof}
 
\begin{corollary}\label{coromain}
Consider $m\in\ene$ and $N/(N+m)<p\leq1$. For each {$f\in \left(C^{m}_c\cap H^p\right)(\erre^N)$} and $\alpha>0$, consider the decomposition $f=g+b$ given by Proposition \ref{decomp_f_g_b}.   
Then, given $\eps>0$ we can rewrite $b=b^\#+\rho$, where 
$b^\#=\sum_{k\in F} \beta_k$ for some finite set $F\subset\ene$, such that $\|\rho\|_{\Lambda^{\gamma_{p}}}\le\eps$, when $N/(N+m)<p<1$, and $\|\rho\|_{bmo}\le\eps$ when $p=1$.
\end{corollary}

\begin{proof}
Let $F$ be the set of indices $k \in \Z$ such that $\ell_k \ge \delta$, where {$0<\delta<1$} is a number to be determined. Note that $F$ is finite because $\O=\, \{x:\ \M f(x)>\alpha\}$ has compact closure. Set $\rho\, :=\, b-b^\#$ and $c_f \; := \displaystyle \max_{0 \leq i \leq m}{\norma{D^{i}f}_{L^{\infty}}}$, since $ f\in C^{m}_c(\R^N)$. 
As the cubes $\{Q_k^*\}$ that contain $\supp \beta_k$ have the bounded intersection property, then the number of cubes with side $\ell_k <{\delta}$ is finite {(and does not depend on $\alpha$)}. 
{Since $N_p+1\leq m$, the inequality \eqref{a4_{m}} implies that 
	$$\|D^{N_p+1}\rho\|_{L^\infty} = \displaystyle \left\|\sum_{\substack{k \in \Z  \\ \ell_k<\delta}} D^{N_p+1}\beta_k\right\|_{L^\infty} \lesssim c_f$$
	and from \eqref{a3_{m}} we have $\|D^i\rho\|_{L^\infty} \lesssim c_f \delta^{m-i}$, for all $i \in \{0,1,...,m-1\}$. In particular, $\|D^{N_p} \rho\|_{L^\infty} \lesssim \delta^{m-N_p} \lesssim \delta$ and $\|\rho\|_{L^\infty} \lesssim \delta^{m} \lesssim \delta^{1-r_{p}}$, since
	$$1-r_p \leq (1-r_p) + \gamma_p = 1+N_p \leq m.$$
	Note that, by Remark \ref{remarkcatarina} and the previous inequalities, we have  
	\begin{eqnarray*}
		[\rho]_{\Lambda^{\gamma_{p}}} &=&  [D^{N_p}\rho]_{\Lambda^{r_p}} \leq \|D^{N_p+1}\rho\|_{L^\infty}\delta^{1-r_p} + \|D^{N_p}\rho\|_{L^\infty}\delta^{-r_p} \lesssim \delta^{1-r_p}.
	\end{eqnarray*}
Thus, $\|\rho\|_{\Lambda^{\gamma_{p}}} = [\rho]_{\Lambda^{\gamma_{p}}} + \|\rho\|_{L^\infty} \lesssim \delta^{1-r_p} = \varepsilon$, for an appropriately chosen $\delta$.
	For the case $p=1$, it suffices to note that $\|\rho\|_{bmo} \leq 3\|\rho\|_{L^{\infty}} \lesssim \eps$}. 
	\qed
\end{proof}


\subsection{Proof of Proposition \ref{czespecial}}

Let  $f\in (H^p\cap C_{c}^{m})(\erre^N)$ and fix a small number $\epsilon>0$.  For each $i \in \Z$, we choose a decomposition $f=g_i+b_{i}=g_i+b_i^\# +\rho_i$ with the properties given by {the Proposition \ref{decomp_f_g_b} and Corollary \ref{coromain}, taking $\alpha_{i}=2^{i}$ and $\varepsilon_i= \epsilon 2^{i}$,} i.e.  $\O^{i} \subseteq \O^{i+1} $ for i $\in \Z_{-}$ with the controls  $|g_{i}(x)|\leq c 2^{i}$,  $\|\rho_{i}\|_{\Lambda^{\gamma_{p}}}\leq \epsilon 2^{i}$ or $\|\rho_{i}\|_{bmo}\leq \epsilon 2^{i}$ when $p=1$, and
$$b_i^\# \displaystyle =\sum_{\substack{k \in \Z  \\ \ell_k\geq\delta_i}} \beta_{ik}, \, \text{where} \, \delta_i = \left( \epsilon 2^{i} \right)^{\frac{1}{1-r_p}}.$$

Given two consecutive values of $i$ and the corresponding decompositions $f=g_{i+1}+b_{i+1}^\# +\rho_{i+1}=g_{i}+b_{i}^\# +\rho_{i}$  
and consider the difference $b_{i}-b_{i+1}=g_{i+1}-g_{i}$.  We see that  
$$\|b_{i}-b_{i+1}\|_{L^\infty}\le \|g_{i+1}\|_{L^\infty}+\|g_{i}\|_{L^\infty} \lesssim 2^{i}.$$ 
{On the other hand, \eqref{a4_{m}} and the bounded intersection property of $\{Q^*_{i,k}\}$ imply that $\|D^j b_i\|_{L^\infty}\le C$, with $C$ independent of $i$, for all $j \in \{1,2,...,m\}$. Consequently,  $\|D^{N_p+1}(b_{i}-b_{i+1})\|_{L^\infty}\le C$. To estimate $D^{N_p} (b_{i}-b_{i+1})$, note that $\delta_i < \delta_{i+1}$, thus
$$b_{i}^\# = \sum_{\substack{k \in \Z  \\ \ell_k\geq\delta_i}} \beta_{ik} = \sum_{\substack{k \in \Z  \\ \ell_k\geq \delta_{i+1} }} \beta_{ik} + \sum_{\substack{k \in \Z  \\ \delta_{i} \leq \ell_k < \delta_{i+1} }} \beta_{ik} \; = \; b_{i+1}^\# + \sum_{\substack{k \in \Z  \\ \delta_{i} \leq \ell_k < \delta_{i+1} }} \beta_{ik} .$$
Thus, \eqref{a3_{m}} implies that 
\begin{align*}
	\norma{D^{N_p} (b_{i}-b_{i+1})}_{L^{\infty}} &\leq \norma{D^{N_p} (\rho_{i}-\rho_{i+1})}_{L^{\infty}} + \norma{D^{N_p} (b_{i}^\#-b_{i+1}^\#)}_{L^{\infty}}\\
	&\lesssim \delta_i^{m-N_p} + \norma{D^{N_p} \left( \sum_{\substack{k \in \Z  \\ \delta_{i} \leq \ell_k < \delta_{i+1} }} \beta_{ik} \right)}_{L^{\infty}}\\
	&\lesssim \delta_i^{m-N_p} + \delta_{i+1}^{m-N_p} \lesssim c_{r_p} \delta_i,
\end{align*}	
for all $i \leq 0$. Again, by Remark \ref{remarkcatarina}, we have
$$\|b_{i}-b_{i+1}\|_{\Lambda^{\gamma_{p}}} \leq \|b_{i}-b_{i+1}\|_{L^{\infty}} + \|D^{N_p+1} (b_{i}-b_{i+1})\|_{L^{\infty}}M_i^{1-r_p} + \|D^{N_p} (b_{i}-b_{i+1})\|_{L^{\infty}} M_i^{-r_p}, $$
for any $M_i>0$. In particular, we can choose $M_i=\delta_i$. Using the previous inequalities
$$\|b_{i}-b_{i+1}\|_{\Lambda^{\gamma_{p}}} \lesssim 2^i + \delta_i^{1-r_p} +\delta_i \delta_i^{-r_p} \simeq 2^i + \epsilon2^i,$$
for all $i \leq 0$.} Thus, we have that 
$$\|g_{i+1}-g_{i}\|_{\Lambda^{\gamma_{p}}}=\|b_{i}-b_{i+1}\|_{\Lambda^{\gamma_{p}}}\le c\,2^{i} ,$$
which implies that the sum $ \sum_{-\infty}^{L} (g_{i+1}-g_{i})$ converges to $0$ in $\Lambda^{\gamma_{p}}(\R^N)$ as $L\to-\infty$. {For the case $p=1$, $\|g_{i+1}-g_{i}\|_{bmo} \lesssim \|g_{i+1}-g_{i}\|_{L^\infty} \lesssim 2^i$, and the sum $ \sum_{-\infty}^{L} (g_{i+1}-g_{i})$ converges to $0$ in $bmo(\R^N)$ as $L\to-\infty$.}

Since $f$ is bounded, $\M f$ is also bounded and there is $N_0\in \Z$ such that $g_{i+1}\equiv g_{i}$ for $i\ge N_0$.
Therefore, we may write
	\begin{align*}
	f=\sum_{-\infty}^{L} (g_{i+1}-g_{i})+
	\sum_{L+1}^{N_0} (g_{i+1}-g_{i})
	&= \sum_{i=L+1}^{N_0} (b^\#_{i}-b^\#_{i+1})+\sum_{i=L+1}^{N_0} (\rho_{i}-\rho_{i+1})
	+\sum_{-\infty}^{L} (g_{i+1}-g_{i})\\
	&= \sum_{i=L+1}^{N_0} 
	\bigg(\sum_{k\in F_{i}}\beta_{i,k}-\sum_{l\in F_{i+1}}\beta_{i+1,l}\bigg)+\sum_{-\infty}^{L} (g_{i+1}-g_{i}) \\
	&:=f^\flat+R,
\end{align*}
where $\|R\|_{\Lambda^{\gamma_{p}}}$  (or $\|R\|_{bmo}$, for $p=1$)} can be made arbitrary small by taking $\epsilon$ small and $L$ sufficiently negative. The next step is to write $f^\flat$ as a finite sum 
\[
f^\flat= \sum_{i=L+1}^{N_0}  \bigg(\sum_{k\in F_{i}}\beta_{i,k}-\sum_{l\in F_{i+1}}\beta_{{i+1},l}\bigg)=\sum_{i=L+1}^{N_0} \sum_k \lambda_{i,k}\, a_{i,k}
\]
where $\left\{a_{i,k}\right\}_{i,k}$ are $H^p$-atoms and $\sum_{i,k} |\lambda_{i,k}|^p\le C\|f\|_{H^p}$.
It is convenient to redefine the functions $\beta_{i,k}$  by setting
$\beta_{i,k}\equiv0$ if $k\notin F_i$ for each fixed $i$, so the approximation 
$f^\flat$ of $f$ may now be written as
\begin{align}
	f^\flat&=\sum_{i=L+1}^{N_0}\bigg(\sum_k (f-P_{f,i,k})\eta_{i,k} - \sum_{\ell} (f-P_{f,i+1,\ell})\eta_{i+1,\ell}\bigg)
	=\sum_{i=L+1}^{N_0}\sum_k A_{i,k}\label{a4}
\end{align}
where
\[
A_{i,k}=(f-P_{f,i,k})\eta_{i,k}-\sum_\ell(f-P_{f,i+1,\ell})\eta_{i+1,\ell}\eta_{i,k}+\sum_{\ell} Q_{f,k,\ell}\,\eta_{i+1,\ell},
\]
with $Q_{f,k,\ell}(x)= \displaystyle P_{ \textcolor{red}{h}, i+1,\ell}$ and $h:=f-P_{f,i+1,\ell}$. It is understood that $P_{f,i,k}=0$ and $Q_{f,k,\ell}=0$ if $k\notin F_i$ or 
$\ell \notin F_{i+1}$. However, in order to prove that \eqref{a4} holds, one needs to observe that, by property (iii) of Lemma \ref{Lema4_M}, $Q_{f,k,\ell} = P_{f,i,k} - P_{f,i+1,\ell}$ and to assume that 
\begin{equation}\label{a5}
	\sum_{k\in F_i}\eta_{i,k}\equiv1, \text{\ on the support of $\eta_{\ell,k+1}$ whenever 
		$\ell\in F_{i+1}$.} 
\end{equation}
This can be achieved inductively: first add a finite number of extra indices to 
$F_{N_0}$ so that  \eqref{a5}  holds for $i=N_0$; then add finitely many additional indices to 
$F_{N_0-1}$  so that  \eqref{a5}  holds for $i=N_0-1$, and so on.
Note that the additional terms $\beta_{i,k}$ added to $f^\flat$ in this way do not increase the norm $\|R\|_{\Lambda^{\gamma_p}}$, which remains at least as small as in the original definition.

Let $B_{i,k}$ be the smallest ball containing $Q^*_{i,k}$.
The same arguments as on pages 108--109 of \cite{St} for the infinite decomposition now show, in the finite decomposition case, that setting
 $a_{i,k}=c^{-1}2^{-i}|B_{i,k}|^{-1/p}A_{i,k}$ and 
$\lambda_{i,k}=c2^i|B_{i,k}|^{1/p}$ we have 
$f^\flat=\sum_{i=L+1}^{N_0}\sum_k \lambda_{i,k}a_{i,k}$. Furthermore,
\[
\sum_{i,k}|\lambda_{i,k}|^p=c\sum_{i,k}2^{ip}|Q_{i,k}|
\le c'\int (\M f)^p(x)\,dx\le C\|f\|_{H^p}.
\]
For each $n=1,2\dots$, we may take $f_n$ be as the function $f^\flat$ in the above construction satisfying $\|R\|_{\Lambda^{\gamma_{p}}}<1/n$. Then, the sequence $\left\{f_n\right\}_{n}$ satisfies the required properties (i) and (ii). \qed

\section{Hardy-Sobolev spaces and elliptic differential operators}\label{S5}

Consider a linear differential operator $A(\cdot,D)$ of order $m \geq 1$ {with smooth complex coefficients, given by 
\begin{equation}\label{opa}
A(x,D)=\sum_{|\alpha|\leq m}a_{\alpha}(x)\partial^{\alpha}: C^{\infty}(\Omega,E) \rightarrow C^{\infty}(\Omega,F),
\end{equation}
where $a_\alpha(x) \in \L(E,F)$ is a linear map for each $x \in \Omega$. We denote by $A^*(x,D)=\sum_{|\alpha|\leq m}a^*_{\alpha}(x)\partial^{\alpha} : C^{\infty}(\Omega,F^*) \rightarrow C^{\infty}(\Omega,E^*)$, with $a^*_\alpha(x) \in \L(F^*,E^*)$ for each $x \in \Omega$, the adjoint operator of $A(\cdot,D)$. Its principal part is denoted by 
\begin{equation}\label{opb}
A_{\nu}(x,D)=\sum_{|\alpha|=m}a_{\alpha}(x)\partial^{\alpha}
\end{equation}
and when $A(\cdot, D) = A_{\nu}(\cdot,D)$, we say that $A$ is a homogeneous operator. 
Since the complex vector spaces $E$ and $F$ are finite-dimensional, for simplicity, we may identify
$E^*$ with $E$, and $F^*$ with $F$ when convenient. 
Naturally the inner products may be regarded as 
$$u \;\cdot_E \;v = \sum_{j=1}^{n_E} u_j\overline{v_j} \quad \text{ and } \quad w \; \cdot_F \;z = \sum_{j=1}^{n_F} w_j\overline{z_j},$$
for all $u=\left(u_1,...,u_{n_E}\right), v=\left(v_1,...,v_{n_E}\right) \in E$ and $w=\left(w_1,...,w_{n_F}\right), z=\left(z_1,...,z_{n_F}\right) \in F$, as well as the complex conjugation. To simplify notation, we omit the space subscripts in the inner products $\cdot_E$ and $\cdot_F$.}

Now suppose $A(\cdot,D)$ is elliptic operator, and consider the order $2m$ differential operator $\Delta_{A}:=A_{\nu}^{\ast}(\cdot,D)\circ A_{\nu}(\cdot,D)$, which may be regarded as an elliptic pseudodifferential operator with symbol in the H\"ormander class $S^{2m}_{1,0}(\Omega \times \R^N; E,E^*)$ (concerning pseudo-differential operators we
refer \cite[Chapter 3]{H}, \cite{Tay} {and \cite{TrPseud}} for more details). So, there exist properly supported pseudo-differential operators $q(\cdot,D) \in {OpS^{{-2m}}_{1,0}(\Omega; E^*,E)}$ {(called parametrix)} and $r(\cdot, D) \in {OpS^{-\infty}(\Omega; E)}$ (regularizing) such that
\begin{equation}\label{parametrix}
u(x)=q(x,D)\Delta_{A}u(x)+{r}(x,D)u(x), \quad \forall \,\, u \in C^{\infty}(\Omega,{E}) .
\end{equation}

{Let $N \geq 2$, {$0<p< N$} and $p^*=\frac{pN}{N-p}$. Theorem 3.1 in \cite{HHP} asserts that for any ball $B \subset \R^{N}$ there exists $C=C(B,p,N)>0$ such that 
\begin{equation}\label{SGNhp}
\|u\|_{h^{p^{*}}}\leq C \|\nabla u\|_{h^{p}}, \quad \forall \, u \in C_{c}^{\infty}(B).
\end{equation} 
Next we present an extension of the previous inequality for higher-order derivatives

\begin{proposition}\label{prop.gradmod}
Assume $N \geq 2$, {$0 <  p  <N$}, and let $A_{\nu}(\cdot,D)$ be an elliptic operators as \eqref{opb}. Then for every point $x_{0} \in \Omega$, there exist an open ball $B=B(x_{0},\ell) \subset \Omega$ and a constant $C=C(B,p,N)>0$ such that 
\begin{equation}\label{jh}
\| u \|_{{\dot{h}^{m,p}}} \leq C\| A_{\nu}(\cdot,D) u \|_{h^{p}}, \quad \forall \, u \in C_{c}^{\infty}(B,E).
\end{equation}
\end{proposition}

\begin{proof} Fix $x_{0} \in \Omega$ and choose $\ell>0$ such that $B=B(x_{0},\ell) \subset \Omega$.
It follows from H\"older's inequality and the Sobolev-Gagliardo-Nirenberg inequality \eqref{SGNhp} in $h^p(B)$ that
\begin{equation}\nonumber
\int_{B}|m_{\varphi}u(x)|^{p}dx \leq |B|^{1-p/p^{*}} \left( \int_{B}|m_{\varphi}u(x)|^{p^{*}}dx \right)^{p/p^{*}} \leq C |B|^{p/N} \|\nabla u\|^{p}_{h^p},
\end{equation}
for all $u \in C_{c}^{\infty}(B,E)$. Bootstrapping the previous argument, we conclude that
\begin{equation}\label{sobolevm}
\| \partial^{\alpha}u \|_{h^p}  \leq C |B|^{(m-|\alpha|)/N} \|u\|_{\dot{h}^{m,p}}, \quad \forall \, u \in C_{c}^{\infty}(B,E)
\end{equation}
for all $|\alpha| \leq m$. Now, for $|\beta| \leq m$, the identity \eqref{parametrix} yields
$$\partial^{\beta}u (x)=\tilde{q}(x,D)[A_{\nu}(x,D)u(x)]+\tilde{r}(x,D)u(x),$$
where $\tilde{q}(x,D):=\partial^{\beta}q(x,D)A_{\nu}^{*}(x,D) \in OpS^{-(m-|\beta|)}_{1,0}(B)$ and $\tilde{r}(x,D):=\partial^{\beta}r(x,D)\in OpS^{-\infty}(B)$. Thus
\begin{align*}
\|\partial^{\beta} u \|_{h^p} \leq \| \tilde{q}(\cdot,D)A_{\nu}(\cdot,D) u \|_{h^p}+ \| \tilde{r}(\cdot,D)u\|_{h^p}
\lesssim \| A_{\nu}(\cdot,D)u \|_{h^p}+\|u \|_{h^p},
\end{align*}
since $\tilde{q}(\cdot,D),\tilde{r}(\cdot,D)$ are bounded operators from $h^{p}$ to itself for {$0<p< \infty$}. 
Using \eqref{sobolevm} and shrinking the radius of $B$ to absorb the second term on the right-hand side, we conclude
the desired estimate.  \qed

\end{proof}	

\begin{remark}\label{obs1}
Let $A(\cdot,D)$ be a linear differential operator as \eqref{opa}, and suppose that all coefficients $a_\alpha(x)$
are bounded at some neighborhood of $x_{0} \in \Omega$, namely $B=B(x_{0}, \ell) \subset\Omega$, and consider $\displaystyle{C:= \sum_{|\alpha|<m}\|a_{\alpha}\|_{L^{\infty}(B)}}$. Assuming {$0 < p <N$,} we have from \eqref{sobolevm}
the following 
\begin{align*}
\|A_{\nu}(\cdot,D)u\|_{h^p} & \leq \|A(\cdot,D)u\|_{h^p}+C\sum_{|\alpha|<m} \| \partial^{\alpha}u\|_{h^p} \\
& \lesssim \|A(\cdot,D)u\|_{h^p}+\sum_{|\alpha|<m}|B|^{(m-|\alpha|)/N} \|u\|_{{\dot{h}^{m,p}}}.  \\
& \lesssim \|A(\cdot,D)u\|_{h^p}+|B|^{1/N} \|u\|_{{\dot{h}^{m,p}}}. 
\end{align*}
Using \eqref{jh} and absorbing the constant by decreasing the radius if necessary, we obtain
\begin{equation}\label{4.8}
\|A_{\nu}(\cdot,D)u\|_{h^p} \leq C  \|A(\cdot,D)u\|_{h^p},  \quad \forall \, u\in C_{c}^{\infty}(B,E),
\end{equation}
for $0< p <N$ and $C=C(B,p)>0$. In particular, taking $A(\cdot,D)=(D^{m})_{|\alpha|\leq m}$ (the non-homogeneous  total derivative operator with order m), this inequality shows that the {quasi-norms} in $\dot{h}^{m,p}(B)$ and $h^{m,p}(B)$ are locally equivalents.  
\end{remark}

\section{The Proof of Theorem A}\label{S6}

Let $\psi\in C_c^\infty(B(0,1))$ with $\int\psi=1$ and $\psi_t(x)\;:=\;t^{-N}\psi(x/t)$, $0<t<1$. 
For simplicity, we take $x_{0}=0$ and $U=B(0,\rho)$  with radius $\rho<1/2$ to be chosen later. Let {$P: \Omega \mapsto E$} be a polynomial of degree at most $m-1$ such that $A(\cdot,D)P= 0$ on $U$; this will be adjusted later. {In fact, by the homogeneity hypothesis, $A(\cdot,D)P= 0$ for all $P \in \P_{m-1}^E$.} Then we may write
\begin{equation}\nonumber
	{\displaystyle{\overline{\psi_t * \left[ A(x,D)\phi \; \cdot \; v \right](x)}}} = \int_{B(x,t)} {A^*(y,D)\left[\psi_{t}\left(x-y \right){v}(y) \right] }\; {\cdot} \; (\phi - P)(y) dy.
\end{equation}
Applying the product rule under the assumption $A^*(\cdot, D)v=0$, and using the 
uniform local boundedness of the functions $\partial^\theta a_\alpha^*$ and $\partial^\theta \psi$ for $0\leq|\theta| \leq m$, with $0<t\leq 1$, we may estimate 
\begin{equation} \nonumber
\left|\psi_t \ast  A(\cdot,D) \phi \cdot v\right|(x)\lesssim \frac{C(\psi)}{t^{N+m}} \int_{B_{x}^{t}} \left( \sum_{0 \leq |\gamma| < m} |\partial^{\gamma}v (y)| \right) |\phi(y) - P(y)| \,dy,
\end{equation}
where $\displaystyle{C(\psi)=\max_{\substack{|\alpha|\leq m\\0\leq \gamma, \theta \leq \alpha}} \sup_{z \in \R^{N}}\left\{\left( 1 + \norma{\partial^{\gamma} a_{\alpha}^*}_{L^{\infty}} \right) |\partial^{\theta}\psi|(z)\right\} }$.

Now, {let $1<\beta<q$ and $\beta'$ be the conjugate exponent of $\beta$, satisfying
 $1<\beta' <p^{*}:=Np/(N-mp)$ if $\frac{N}{N+m}<p< \min \left\{ 1, \frac{N}{m} \right\}$ or $1<\beta' <p$  if $1<p \leq N/m$ for $m<N$. Thus}
\begin{align*}
\left|\psi_t \ast A(\cdot,D) \phi\cdot v\right|(x)\lesssim 
\left[ \fint_{B_{x}^{t}} \left( \sum_{0\leq|\gamma| < m} \left| \partial^{\gamma}v(y) \right| \right)^{\beta}dy \right]^{{1}/{\beta}} \left[ \fint_{B_{x}^{t}} \left|\frac{1}{t^{m}}(\phi(y) - P(y))\right|^{\beta'}  dy \right]^{{1}/{\beta'}}
\end{align*}
which implies that $\displaystyle{\sup_{0<t<1}\left|\psi_t \ast A(\cdot,D) \phi\cdot v\right|^{r}(x)}$ is bounded by
\begin{equation*}
	C(\psi) \left[ M\left(\left( \sum_{0\leq|\gamma| < m} \left| \partial^{\gamma}v \right| \right)^{\beta}\right)(x) \right]^{{r}/{\beta}} \left[ \sup_{0<t<1} \fint_{B_{x}^{t}} \left|\frac{1}{t^{m}}\left[\phi(y) - P(y)\right]\right|^{\beta'}  dy \right]^{{r}/{\beta'}},
\end{equation*}
where $M$ denotes the Hardy-Littlewood maximal operator. Integrating on $\R^{N}$ and using H\"older's inequality with exponents $qr^{-1}$ and $pr^{-1}$ (since $1/r=1/p+1/q$), we obtain the estimate 
\begin{align}\label{mainestimate}
	\norma{A(\cdot,D)\phi \cdot v}^{r}_{h^r}
	\displaystyle{ \lesssim  \|M \left( D^{m}_{<}v \right) \|^{{r}/{\beta q}}_{L^{{q}/{\beta}}}
	\Bigg(\int_{\R^{N}}\left[ \sup_{0<t<1} \fint_{B_{x}^{t}} \left|\frac{1}{t^{m}}\left[\phi(y) - P(y)\right]\right|^{\beta'}  dy \right]^{{p}/{\beta'}} dx\Bigg)^{{r}/{p}}}
\end{align}
where $\displaystyle{ D^{m}_{<}v=\left( \sum_{0\leq|\gamma| < m} \left| \partial^{\gamma}v \right| \right)^{\beta}}$. 
Since the operator $M$ is bounded on $L^{{q}/{\beta}}$ for $1<\beta<q$, 
the first factor on the right-hand side can be estimated by
\begin{eqnarray}\label{ma}
	\|M \left( D^{m}_{<}v \right) \|^{{r}/{\beta q}}_{L^{{q}/{\beta}}}\leq
	 \norma{ \left( \sum_{0\leq|\gamma| < m} \left| \partial^{\gamma}v \right| \right)^{\beta} }_{L^{{q}/{\beta}}}^{{1}/{\beta}}
	=\norma{ \sum_{0\leq|\gamma| < m} \left| \partial^{\gamma}v \right| }_{L^{q}} \lesssim  \|v\|_{W^{m-1,q}}. .
\end{eqnarray}
In order to conclude the proof of Theorem A, we need to estimate 
the remaining factor on the right-hand side of (6.1). This follows from a fundamental inequality for local Hardy-Sobolev spaces, and is an ersatz for  \cite[Lemma II.2]{CLMS}:
	
\begin{lemma}\label{lemanovo}
Let $m \in \N$, $\frac{N}{N+m}< p \leq { \min \left\{ 1, \frac{N}{m} \right\} }$, and $\alpha \in [1,p^{*})$ with $p^*=Np/(N-mp)$ (in particular, if $p = \frac{N}{m}$ we take $p^*=\infty$). Then there exist a constant $C >0$ and a polynomial $P_{x,t,f} \in \P_{m-1}$ such that the estimate 
\begin{align}\label{5.2}
\int_{\R^{N}} \bigg[ \sup_{0<t<1}\fint_{B_{x}^{t}} \bigg\{\frac{1}{t^m} |f(y)- P_{x,t,f} (y)| \bigg\}^{\alpha}dy \bigg]^{{p}/{\alpha}}dx \leq C \|f\|_{h^{m,p}}^p
\end{align}
holds for every $f\in C_{c}^{\infty}(\R^{N})$. {Moreover, the same conclusion holds when {$1<p<\infty$} for $m<N$ and $\alpha \in [1,p)$.}
\end{lemma}  
Note that 
\begin{align*}
\int_{\R^{N}}\left[ \sup_{0<t<1} \fint_{B_{x}^{t}} \left|\frac{1}{t^{m}}\left[\phi(y) - P(y)\right]\right|^{\beta'}  dy \right]^{{p}/{\beta'}} dx \lesssim \sum_{j=1}^{n_E} \int_{\R^{N}}\left[  \sup_{0<t<1} \fint_{B_{x}^{t}} \left|\frac{1}{t^{m}}\left[\phi_j(y) - P_j(y)\right]\right|^{\beta'}  dy \right]^{{p}/{\beta'}} dx.
\end{align*}
Applying Lemma \ref{lemanovo} to each component in the sum above, with $P_j:=P_{x,t,\phi_j}$, and using Proposition \ref{prop.gradmod} and Remark \ref{obs1}, we obtain a ball $B_{0}^{\rho}=B(0,\rho)$ (for $\rho$ sufficiently small) and $C=C(B,p,N)>0$ such that 
\begin{equation}\label{eq3.9}
\left(\int_{\R^{N}}\left[ \sup_{0<t<1}\fint_{B_{x}^{t}} 
\left\{\frac{1}{t^m}\left|{\phi}(y)-P_{x,t,{\phi}} (y)\right|\right\}^{\alpha}   dy \right]^{{ p}/{\alpha}}dx\right)^{{1}/{p}} \leq C(B_{0}^{\rho}) \|A(\cdot, D){\phi}\|_{h^{p}},
\end{equation}
for all $\phi \in C_{c}^{\infty}(B_{0}^{\rho})$. Plugging the inequalities \eqref{eq3.9} and \eqref{ma} into \eqref{mainestimate} we conclude the proof of Theorem A.  
\begin{flushright}
	\qed
\end{flushright}

The remainder of this section is devoted to proving Lemma \ref{lemanovo}. We point out that the second part of the result is trivially satisfied. 
Let $1<p<\infty$,  $m<N$, and $\alpha \in [1,p)$. From the generalized Poincar\'e inequality  \eqref{Desigualdade_Poincarè_M_maior} for $T=1$, choosing $\alpha<\gamma<p$ (in fact, we may take $\gamma=\alpha$ if $\alpha>1$), there exist a constant $C=C_{N,p, m, \alpha} >0$ and a polynomial $P_{x,t,f} \in \P_{m-1}$ such that
	$$\left( \fint_{B(x,t)} \left| \frac{1}{t^m}\left[f(y) - P_{x,t,f}(y)\right] \right|^{\alpha} dy \right)^{{1}/{\alpha}} \leq C \left( \fint_{B(x,t)} \left| D^m f(y) \right|^{\gamma} dy \right)^{{1}/{\gamma}}$$
for all  $f\in \ccinf(\R^{N})$. 
Now, taking the supremum for all $0<t<1$ we get 
	\begin{align*}
		\int_{\R^{N}} \bigg[ \sup_{0<t<1}\fint_{B_{x}^{t}} \bigg\{\frac{1}{t^m} |f(y)- P_{x,t,f} (y)| \bigg\}^{\alpha}dy \bigg]^{{p}/{\alpha}}dx 
		&\lesssim \norma{\left[ M\left( |D^m f|^\gamma \right) \right]^{{1}/{\gamma}} }_{L^p(\R^N)}^p \\
		&= \norma{ M\left( |D^m f|^\gamma \right)}_{L^{{p}/{\gamma}}(\R^N)}^{{p}/{\gamma}}\\
		& \lesssim  \norma{ |D^m f|^\gamma }_{L^{{p}/{\gamma}}(\R^N)}^{{p}/{\gamma}} \\
		&= \norma{ D^m f }_{L^{p}(\R^N)}^{p} 
	\end{align*}
for all  $f\in \ccinf(\R^{N})$, where in the last inequality uses the boundedness of maximal Hardy-Littlewood operator in $L^{p/\gamma}$ since $p>\gamma$. 
Inequality \eqref{5.2} then follows, recalling
that $h^{p}(\R^{N})$ is equivalent to $L^{p}(\R^{N})$ with comparable norms when $1<p<\infty$. 

\begin{remark}
{In the particular case $A(\cdot,D)=D^{m}$ (the total derivative operator) with $m<N$, which is elliptic, the upper bound in the Theorem A coincides with the obtained in \cite{CLMS}, namely $\frac{N}{N+m}<p<\infty$.} 
\end{remark}

\begin{proof} 
{Let $\frac{N}{N+m}< p \leq { \min \left\{ 1, \frac{N}{m} \right\} }$ and $\alpha \in [1,p^{*})$.}  
For each $f\in C_{c}^{\infty}(\R^{N})$, $x \in \R^{N}$, and $0<t<1$, consider  $P_{x,t,f}(y) := \; \inner{f}{ \phi_{Q(x,2t)}(y,\cdot)}$ the polynomial of degree $m-1$ given by the Lemma \ref{Lema4_M}, and define the sublinear operator
$$Tf(x) = \sup_{0<t<1}  \left\{ \fint_{B(x,t)} \left| \frac{1}{t^m}\left[f(y) - P_{x,t,f}(y)\right] \right|^{\alpha} dy \right\}^{{1}/{\alpha}}.$$
Rewriting the inequality \eqref{5.2}, it is sufficient to prove that 
\begin{equation}\label{5.3a}
\int_{\R^{N}}\left|Tf(x)\right|^{p}dx
 \leq C \|f\|_{h^{m,p}}^{p},
\quad \forall \, f\in \ccinf(\R^{N}).
\end{equation}

Consider the Bessel potential operator $J_{m}:=(1-\Delta)^{-m/2}$ for $m \in \N$, which can be understood as a pseudo-differential operator with symbol $\left(1+|\xi|^{2}\right)^{-m/2} \in S^{-m}(\R^N)$. Denote the sublinear operator $T' \,:= \,T \circ J_{m}$. Writing $f= J_{m}g$ and $D^m f=(D^m \circ J_{m})g$, in which
$$
\widehat{g}(\xi) := (1+|\xi|^2)^{\frac{m}{2}} \hat f(\xi) = \frac{1}{(1+|\xi|^2)^{\frac{m}{2}}}\hat f(\xi) + \sum_{j=1}^{m}  \binom{m}{j} \dfrac{|\xi|^{2j}\hat{f}(\xi)}{(1+|\xi|^2)^{\frac{m}{2}}}, 
$$
we see that  $ \sum_{j=0}^{m}\norma{D^j f}_{h^p} \simeq\|g\|_{h^{p}}$, because pseudodifferential operators with negative order are bounded from $h^{p}(\R^N)$ to itself for $0<p\leq 1$.
Thus, to obtain  \eqref{5.3a} it suffices to prove that there exists an uniform constant  $C>0$ such that
for every compact subset $K\subset\subset\erre^N$, the control
{\begin{equation}\label{5.3}
\int_{K}\left|T'g(x)\right|^{p}dx
 \leq C \|g\|_{h^{p}}^p,
\quad \forall \, g \in \ccinf(\R^{N})
\end{equation}  
holds. Indeed, assume \eqref{5.3} holds for all $K \subset \subset \R^N$, and set $h(x) := \left|T'g(x)\right|^{p}$ and $h_n(x) := h(x) \mathcal{X}_{B[0,n]}(x)$ for each $n \in \N$. Then, by Fatou Lemma, we have 
\begin{equation*}
	\int_{\R^N}\left|T'g(x)\right|^{p}dx=\int_{\R^N}h(x)dx \leq \limsup_{n} \int_{\R^N}h_{n}(x)dx    \leq C \|g\|_{h^{p}}^p
\end{equation*}  
and hence
	\begin{eqnarray*}
		\int_{\R^N}\left|Tf(x)\right|^{p}dx = \int_{\R^N}\left| T' g(x)\right|^{p}dx \leq
	  C \|g\|_{h^{p}}^p \lesssim \|f\|^{p}_{h^{m,p}}
	\end{eqnarray*}
which proves \eqref{5.3a}.}

The proof of \eqref{5.3} proceeds in two fundamental steps. The first is the reduction of the estimate to the case where $g$ is a standard bounded  $h^p$-atom, (see Subsection \ref{atoms}), i.e. we show that there exists a uniform constant $C>0$ such that 
\begin{equation}\label{atm}
\int_{K}\left|T'a(x)\right|^{p}dx
 \leq C 
\quad \forall \, a \, h^{p}-\text{atom}.
\end{equation}
This point is crucial, because there exist examples of linear operators that are uniformly bounded on standard atoms but cannot be extended to Hardy spaces (see \cite{B}). In the same spirit as \cite[Appendix A1]{HHP}, we state the following result: 
\end{proof}

\begin{lemma}\label{6.2}
Given $f\in\ccinf(\erre^N)$, there exist a sequence of functions $f_n\in \left(C^{m}_{c} \cap h^{p}\right)(\R^N)$ and a uniform constant $C>0$  such that
\begin{enumerate}
        \item [(i)] $\|f_n\|_{h^p}\le C \|f\|_{h^p}$;
	\item [(ii)] $\displaystyle \lim_{n\to\infty} \norma{T'(f_n-f)}_{L^p(K)} = 0,$ {for all $K \subset \subset \R^N$};
	\item [(iii)] $\norma{T'f_n}_{L^p(\R^N)} \le C \norma{f}_{h^p}$.
\end{enumerate}
\end{lemma}

The estimate \eqref{5.3} follows directly from (iii) above. 

\begin{proof} By Proposition \ref{czespecial}, given $f\in C_{c}^{\infty}(\erre^N)$ there exists a sequence of functions $f_n\in \left(C^{m}_c \cap h^{p}\right)(\erre^N)$ such that (i) holds and 
$\displaystyle\lim_{n\to\infty}\|f-f_n\|_{\Lambda^{\gamma_{p}}}=0$. We also claim that 
\begin{equation}\label{3}
\displaystyle{\lim_{n\to\infty}T'(f_n-f)(x)=0}
\end{equation}
uniformly on compact subsets $K \subset \erre^N$,
which directly implies (ii).
To prove \eqref{3}, it suffices to show that the sublinear operator
\begin{align*}
T'f(x) = \sup_{0<t<1} \left( \fint_{B_x^t} 
\left\{ \frac{1}{t^m} \left| 
J_mf(y) - P_{x,t,J_mf}(y) \right|\right\}^\alpha dy \right)^{1/\alpha}
\end{align*}
satisfies $\displaystyle{\sup_K T'f(x)\le C \|f\|_{\Lambda^{\gamma_{p}}}}$.
Applying the inequality \eqref{Desigualdade_Poincarè_M_maior} for $J_{m}f$ and using  the boundedness of the order-zero 
pseudodifferential operators $D^m \circ J_m$ on H\"older spaces (see \cite{Beals}), we obtain
\begin{align*}
	T'f(x)	\le \|D^m \circ J_mf\|_{L^\infty}\le \|D^m \circ J_mf\|_{\Lambda^{\gamma_{p}}} \le C \|f\|_{\Lambda^{\gamma_{p}}}
\end{align*}
for all $x \in K \subset \subset \R^N$. 
For a moment, we assume the validity of  \eqref{atm} in order to  prove (iii). 
Since $f_n \in C^{m}_c(\erre^N)$, we can decompose $f_n$ as a   
\textit{finite} sum of smooth bounded $h^p$-atoms, i.e. $f_n=\sum_{k=1}^{k_n}\lambda_{k,n} a_{k,n}$ satisfying $\sum_k |\lambda_{k,n}|^p\le C\|f\|_{h^p}$ 
with $C>0$ independent of $k$ and $n$ (see the discussion in Section \ref{S4}). Hence,
\begin{align*}
\|T'f_n\|_{L^p}^p\le \sum_{k=1}^{k_n}|\lambda_{k,n}|^p \|T'a_{k,n}\|_{L^p}^p
  \le C  \|f\|^p_{h^p}.
\end{align*}
\qed

\end{proof}

\subsection{Proof of the uniform control \eqref{atm} }
 
Let $a$ be an $h^{p}$-atom compactly supported in a ball $B=B(x_{0},r)$ 
as defined in Subsection \ref{atoms}. Without loss of generality, we may assume $x_{0}=0$. 

Consider first $|x|\leq 10r$ and set $b \,:= \, J_{m} a$. 
Inequality  \eqref{Desigualdade_Poincarè_M_maior} gives 
\begin{equation}
	\left( \fint_{B_x^t} \left|\frac{1}{t^m} \left[b(y) -  P_{x,t,b}(y)\right]\right|^{\alpha} dy \right) ^{{1}/{\alpha}} \lesssim \left( \fint_{B_x^t} | D^m b(y)|^{\alpha} dy \right) ^{{1}/{\alpha}}.
\end{equation}
Taking the supremum over $0<t<1$ yields 
$$T'a(x) \lesssim \left[M(|D^m b|^{\alpha})(x)\right] ^{{1}/{\alpha}}.$$
Hence, if $\alpha< \beta < \infty$ then
\begin{align*}
\|M(|D^mb|^{\alpha})^{1/\alpha}\|_{L^{\beta}}=\|M(|D^mJ_{m}a|^{\alpha})^{1/\alpha}\|_{L^{\beta}}  \lesssim \||D^m J_{m}a|^{\alpha}\|^{1/\alpha}_{L^{\beta/\alpha}}= \|D^m J_{m}a\|_{L^{\beta}}&\lesssim \|a\|_{L^{\beta}} \\
 &\lesssim r^{N\left(\frac{1}{\beta}-\frac{1}{p}\right)},  
\end{align*}
where the second inequality follows from the fact that $D^m \circ J_{m}$ is a pseudo-differential operator of order zero and hence bounded on $L^{\beta}$ for any
$1<\beta<\infty$. Thus,
\begin{align*}
\displaystyle{\left(\int_{|x|\leq 10r}|T'a(x)|^{p}dx\right)^{1/p}} &\displaystyle{ \lesssim r^{N\left(\frac{1}{p}-\frac{1}{\beta}\right)}\left(\int_{|x|\leq 10r}|T'a(x)|^{\beta}dx\right)^{1/\beta}} \\
& \displaystyle{ \lesssim r^{N\left(\frac{1}{p}-\frac{1}{\beta}\right)} \|M(|D^mb|^{\alpha})^{1/\alpha}\|_{L^{\beta}} \lesssim 1},
\end{align*}
uniformly for any $h^{p}-$atom. We point out that only the size condition of the atom is required in the previous control. 
 
Now consider the region $|x|\ge 10 r$. We distinguish atoms according 
to the size of their support $B(0,r)$ as in Subsection \ref{atoms}: atoms with $r\leq 1$ are $H^{p}-$atoms 
and satisfy moment conditions, while atoms with $r>1$ require no moment condition. 

Recall that the tempered distribution kernel $K_{m}$ of Bessel potential $J_{m}$, seen as a pseudodifferential operator, is of convolution type, belongs to $L^{1}(\R^{N})$, smooth off the diagonal, and satisfies the estimate
(see e.g. \cite{AH, St2})
\begin{equation}\label{k1}
\sup_{|\alpha|+|\beta|=M}{\left|\partial_{y}^{\alpha}\partial_{z}^{\beta}K_m(y-z)\right|}\leq C_{\alpha, \beta}|y-z|^{-(M+N-m)},\;\;\;y \neq z
\end{equation}
for {$N-m+M>0$} and  
 \begin{equation}\label{k2}
\sup_{|y-z|\geq 1/2}{|y-z|^{L}\left|\partial_{y}^{\alpha}\partial_{z}^{\beta}K_m(y-z)\right|}\leq C_{\alpha, \beta,L},\;\;\; \forall \;\alpha, \beta \in \N^{n},\,
\end{equation}
for any $L \in \N$. 
If an $h^{p}$-atom $a(z)$ is supported in $B(0,r)$
with $r \geq 1$, a consequence from \eqref{k2} is  
that for every positive integer $L$, there exists a constant $C_{N,m,L}>0$ such that
\begin{equation}\label{nucleo}
\sup_{w\in B_{x}^{t}}|D^m b(w)|\leq C_{N,m,L}\; r^{-\frac{N}{p}+N}|x|^{-L},\;\;\; \forall\; 0<t<1\; \text{and}\; |x|> 10r.
\end{equation}
Indeed, for $y \in B(x,t)$ and $0<t<1$, we have
$|y-x|<t<1\leq r <\frac{1}{10}|x|$  which implies $|y|>|x|-|x-y|>\frac{9}{10}|x|.$ 
Then, for $|w|<r$ follows $|y-w|\geq |y|-|w|>\frac{8}{10}|x|$. Hence, for $y \in B(x,t)$ we have
\begin{align*}
\left|D^{m}b(y)\right|=\left| \int_{B(0,r)}D^{m}{K_{m}(y-w)}a(w)dw\right| &\leq C_{0,m,L}|B(0,r)| \|a\|_{L^{\infty}}|x|^{-L} \\
&\leq C_{N,m,L}r^{-\frac{N}{p}+N}|x|^{-L},
\end{align*}
for all $y \in B_{x}^t$. Therefore, using inequality \eqref{Desigualdade_Poincarè_M_maior} again we obtain
\begin{align*}
	T'a(x) &= \left[ \sup_{0<t<1}  \left\{ \fint_{B(x,t)} \left| \frac{1}{t^m}\left[b(y) - P_{x,t,b}(y)\right]\right|^{\alpha} dy \right\} \right]^{{1}/{\alpha}}\\
	&\lesssim \left[ \sup_{0<t<1}  \left\{ \fint_{B(x,t)} \left|D^{m}b(y)\right|^{\alpha} dy \right\} \right]^{{1}/{\alpha}} \lesssim \dfrac{r^{-\frac{N}{p}+N}}{|x|^L}.
\end{align*}
This implies 
\begin{eqnarray*}
	\int_{|x|\geq10r} |T'a(x)|^p dx= r^N \int_{|x|\geq10} |T'a(rx)|^p dx
	& \lesssim &  r^{(N-L)p}  \int_{|x|\geq10} \frac{1}{|x|^{Lp}} dx,
\end{eqnarray*}
and the estimate is uniformly bounded for $r \geq 1$, by choosing $L>N/p$.

Now assume that $a(z)$ is an $h^{p}-$atom supported in $B(0,r)$ with $r<1$ that satisfies the moment condition $\int_{\R^{N}}a(z)z^\beta dz=0$, for all $|\beta| \leq m-1$. {Note that we may always assume additional moment conditions, namely  $|\beta| \leq N_{p}:= \lfloor N\left(\frac{1}{p}-1\right) \rfloor$, since $N_p \leq \gamma_{p}<  m-1$}. 
Fix $w_0$ and let $P_{w_o}$ denote the Taylor polynomial of degree $m-1$ of $K_m$ centered at $w_0$. Then
$$\left| K_m(w) - P_{w_0}(w) \right| \lesssim |w-w_0|^{m}  
 \sup_{|\gamma|=m} \left|D^{\gamma}K_m(\widetilde{w_0})\right|,$$
for some $\widetilde{\omega_{0}}\in [w,w_{0}]$.
For each $y \in B(0,10r)^c$, consider $w_0:=y$ and $w:=y-z$ with $z \in B(0,r)$ then by 
\eqref{k1} taking $M=m$ and reproducing \eqref{nucleo}, we obtain
\begin{align}\label{5.6}
	 \displaystyle{|b(y)|} &= \left| \int_{B(0,r)}a(z)\left[K_m(y-z)-P_{y}(y-z)\right]dz \right| \notag\\
		&\lesssim \displaystyle{ \int_{B(0,r)}|a(z)|z|^m \sup_{w \in [y-z, y]} | D^m_y K_m(w)| dz}
 \lesssim \displaystyle{\frac{r^{-\frac{N}{p}+N+m}}{|y|^{N}}},
	\end{align} 
for all $|y|>10r$. Similarly, for each $|\beta|=j \in \N$ 
\begin{equation}\label{5.7}
|D^\beta b(y)|  \leq C \; \frac{r^{-\frac{N}{p}+N+m}}{|y|^{N+j}}, \quad |y|\geq 10r. 
\end{equation} 
{Next, we claim the global estimate
	\begin{equation}\label{eq_geral_global}
		|b(y)|:=|J_ma(y)|\lesssim r^{-\frac{N}{p}+m},\quad \text{ for all } y\in \R^{N}.
	\end{equation}
First, assume $1\leq m<N$. Then, the control \eqref{k1} with $M=0$ implies that 
$|J_ma(y)|\lesssim I_{m}|a|(y)$, for all $y \in \R^{N}$, 
where  $I_{\gamma}f:=f\ast|y|^{-(N-\gamma)}$ is the Riesz Potential (up to a constant factor). 
By a well-known estimate (see \cite[pp. 354]{St}), if 
	$1<s<N$ and $s^*=Ns/(N-m s)$ then 
	\[
|J_ma(y)|\lesssim I_{m}|a|(y) \leq C [M|a|(y)]^{s/s^{*}}\|a\|^{1-s/s^{*}}_{L^{s}} \leq C\norma{a}_{L^{\infty}} \; r^m \leq Cr^{-\frac{N}{p}+m}, \quad \text{ for all } y\in \R^{N}.
	\] 
	Now suppose $N \leq m$. Applying Sobolev-Gagliardo-Nirenberg inequality in $L^1(B(0,r))$ for $j:=m-N+1$ times, and use control \eqref{k1} with $M=j$, we obtain  
		\begin{eqnarray*}
			|b(y)| \leq \int_{B(0,r)} |K_m(y-z) a(z)| dz &\lesssim& r^{-\frac{N}{p}} \int_{B(0,r)} |K_m(y-z )| dz\\
			& \lesssim & r^{j-\frac{N}{p}} \int_{B(0,r)} {|D^j K_m(y-z)|} dz \\
			& \lesssim & r^{j-\frac{N}{p}} \int_{B(0,r)} \frac{1}{|y-z|} dz\\
			& = & r^{j-\frac{N}{p}} I_{N-1}(\X_{B(0,r)})(y).
		\end{eqnarray*}
By the same reasoning as before, we can estimate $|I_{N-1}(\X_{B(0,r)})(y)| \leq C \| \X_{B(0,r)} \|_{L^\infty} \; r^{N-1}$ and therefore 
		$$|b(y)| \leq C r^{j-\frac{N}{p} +N -1} = C r^{-\frac{N}{p} +m}, \quad \text{ for all } y\in \R^{N}.$$}
Combining estimate \eqref{eq_geral_global} with \eqref{5.6} yields the global estimate
	\begin{equation}\label{5.8}
		|b(y)|=|J_ma(y)|\leq C \frac{r^{-\frac{N}{p}+N+m}}{(r+|y|)^{N}},\quad y\in \R^{N},
	\end{equation}
and consequently 
\begin{equation}\label{5.9}
	|b(ry)|\leq C \frac{r^{-\frac{N}{p}+m}}{(1+|y|)^{N}},\quad y\in \R^{N}.  
\end{equation}

Consider now $|x|\geq 10$. Then,
\begin{eqnarray*}
	T'(a)(rx) &=& \sup_{0<t<1} \left\{  \left[ \fint_{B(rx,t)} \left| \frac{1}{t^m}[b(\tilde{y}) - P_{rx,t,b}(\tilde{y})]\right|^{\alpha} d\tilde{y} \right] ^{{1}/{\alpha}} \right\}\\
	&=& \sup_{0<t<1} \left\{  \left[ \fint_{B\left(x,\frac{t}{r}\right)} \left| \frac{1}{t^m}[b(ry) - P_{rx,t,b}(ry)]\right|^{\alpha} dy \right] ^{{1}/{\alpha}} \right\}\\
	&\leq& \underbrace{\sup_{0<t<\frac{r|x|}{2}} \left\{  \left[ \fint_{B\left(x,\frac{t}{r}\right)} \left| \frac{1}{t^m}[b(ry) - P_{rx,t,b}(ry)]\right|^{\alpha} dy \right] ^{{1}/{\alpha}} \right\}}_{(I)}\\
	&+& \underbrace{\sup_{\frac{r|x|}{2}<t<1} \left\{  \left[ \fint_{B\left(x,\frac{t}{r}\right)} \left| \frac{1}{t^m}[b(ry) - P_{rx,t,b}(ry)]\right|^{\alpha} dy \right] ^{{1}/{\alpha}} \right\}}_{(II)}.
\end{eqnarray*}
 Analogously to the previous arguments, we get
 \begin{equation} \label{eq_A}
 	(I) \lesssim \dfrac{r^{-\frac{N}{p} }}{|x|^{N+m}}
 \end{equation}
and
\begin{eqnarray*}
	(II) &=& r^{-m} \; \sup_{\frac{|x|}{2}<t<\frac{1}{r}} \left\{  \left[ \fint_{B\left(x,t\right)} \left| \frac{1}{t^m}[b(ry) - P_{rx,rt,b}(ry)]\right|^{\alpha} dy \right] ^{{1}/{\alpha}} \right\}\\
	&\lesssim & r^{-m} \; \sup_{\frac{|x|}{2}<t<\frac{1}{r}} \left\{ \dfrac{1}{t^{m + \frac{N}{\alpha}}} \left( \norma{b(r\cdot) }_{L^{\alpha} (B(x,t))} + \norma{ P_{rx,rt,b}(r\cdot)}_{L^{\alpha} (B(x,t))} \right) \right\}.
\end{eqnarray*}	
Since $\displaystyle{|P_{z, s,f}(y)| \leq C\fint_{Q(z,2s)} |f(z)| dz}$ for all $y \in B(z,s)$,  inequality \eqref{5.9} gives 
\begin{equation}\label{ba}
\norma{b(r\cdot) }_{L^{\alpha} (B(x,t))} \lesssim r^{m -\frac{N}{p}}
\end{equation}
and
$$\norma{ P_{rx,rt,b}(r\cdot)}_{L^{\alpha} (B(x,t))} \lesssim \;\dfrac{r^{m - \frac{N}{p}}}{t^{N(1-\frac{1}{\alpha})}} \; [1+\log t].$$		
Hence,
\begin{equation}\label{eq_B}
	(II)	\lesssim  r^{-\frac{N}{p}}  \left[ \dfrac{1}{|x|^{m + \frac{N}{\alpha}}} + \dfrac{\log |x|+1}{|x|^{m+N}} \right].
\end{equation}
Therefore, combining inequalities \eqref{eq_A} and \eqref{eq_B}, we obtain 
\begin{eqnarray*}
	\int_{|x|\geq 10r} |T'(a)(x)|^p dx &=& r^N \int_{|x|\geq 10} |T'(a)(rx)|^p dx\\
	&\lesssim& \int_{|x| \geq 10} \left( \dfrac{1+\log |x|}{|x|^{N+m}} + \dfrac{1}{|x|^{m + \frac{N}{\alpha}}} \right)^p dx< \infty,
\end{eqnarray*}
{since $\alpha <p^*$}. Finally, for $\alpha = 1$ we may choose  $\gamma>1$ and then by inequality \eqref{Desigualdade_Poincarè_M_maior} we have
$$T'(a)(x) \lesssim \left[M(|D^m b|^{\gamma})(x)\right] ^{{1}/{\gamma}},$$
and the proof proceeds \textit{mutatis mutandis}, except that \eqref{ba} is replaced by
$$\norma{(b(r\cdot) }_{L^{1} (B(x,t))} \lesssim r^{m -\frac{N}{p}} \;\log t$$
and 
$$\norma{ P_{rx,rt,b}(r\cdot)}_{L^{1} (B(x,t))} \lesssim r^{m - \frac{N}{p}} \; [1+\log t].$$

\section{Non-homogeneous estimates and applications}\label{S7}

We start by presenting a non-homogeneous version of Theorem A in which the assumption $A^*(\cdot, D)v = 0$ (in the sense of distributions) is removed.  

\begin{alphatheo}\label{teo_principalb}
	Let $A(\cdot,D)$ be an elliptic {homogeneous} linear differential operator {with order m} as before on $\Omega \subset \R^N$, and let $r, p,q$ satisfy
	$$\dfrac{N}{N+m} < p \leq {\frac{N}{m}}, \; \; \; 1 < q \leq \infty \; \; \text{and} \;\; \dfrac{1}{r} \; := \; \dfrac{1}{p} + \dfrac{1}{q} < 1 + \dfrac{m}{N}.$$
	Then for each $x_0 \in \Omega$ there exist an open neighborhood $x_0 \in U \subset \Omega$ and a constant $C>0$ such that
	\begin{equation*}
		\norma{A(\cdot,D)\phi \cdot v}_{h^r} \leq C \norma{A(\cdot,D)\phi}_{h^p}  \left( \|v\|_{W^{m-1,q}} +\|A^*(\cdot, D)v\|_{L^q} \right), 
	\end{equation*}
	for any $\phi \in C^{\infty}_c(U,E)$ and $v \in C^{\infty}_c(U, F)$. 	
\end{alphatheo}

\begin{proof}
Let $\psi\in C_c^\infty(B(0,1))$ with $\int\psi=1$ and $\psi_t(x)\;:=\;t^{-N}\psi(x/t)$, $0<t<1$. Following the notation of Theorem A, we may write
\begin{equation}\nonumber
\displaystyle{{\psi_t * \left[ A(\cdot,D)\phi \cdot v \right](x)}} = \int_{B(x,t)} {A^*(y,D)\left[\psi_{t}\left(x-y \right) \overline{v(y)} \right] }(\phi - P)(y) dy,
\end{equation}
where 
$P$ is a polynomial of degree at most $m-1$. 
By the uniform local boundedness of the functions $\partial^\theta a_\alpha^*$ and $\partial^\theta \psi$ for $0\leq|\theta| \leq m$, and taking $0<t\leq 1$ we have
\begin{equation} \nonumber
\left|\psi_t \ast  A(\cdot,D) \phi \cdot v\right|(x)\lesssim \frac{C(\psi)}{t^{N+m}} \int_{B_{x}^{t}} \left( \sum_{0 \leq |\gamma| < m} |\partial^{\gamma}v (y)| +|A^{\ast}(y,D)v(y)|\right) |\phi(y) - P(y)| \,dy,
\end{equation}
and the conclusion follows \textit{mutatis mutandis} as before. 	\qed

\end{proof}

\subsection{Elliptic system of vector fields }\label{aplica}

Consider $n$ complex vector fields $L_{1},\dots,L_{n}$, $n \geq 2$, with smooth coefficients defined on $\Omega \subseteq \R^{N}$, where $N \geq 2$. We will assume that  the system of vector fields $\L:=\{L_1,\dots,L_n\}$ is linearly independent. Consider
the gradient 
$\nabla_{\L}:C^{\infty}(\Omega)\mapsto C^{\infty}(\Omega,\mathbb{C}^n)$ given by  
$$\nabla_{\L}\,u\,:= \;(L_{1}u,..,L_{n}u), \quad \quad u \in C^{\infty}(\Omega)$$ 
and its formal {complex}  adjoint operator, defined for $v \in C^{\infty}(\Omega,\mathbb{C}^{n})$ by
\begin{equation}\nonumber
{\rm div}_{\L^{*}}\,v:=L^{*}_{1}v_{1}+...+L^{*}_{n}v_{n}.
\end{equation}
 Here $L_{j}^{*}=\overline{ L_j^{t}}$ for $j=1,\dots,n$, where $\bar L_{j}$ denotes the vector field obtained from $L_j$ by conjugating its coefficients, and $L_j^t$ is the formal transpose of $L_j$. We say the system $\left\{L_{1},\dots,L_{n}\right\}$ is elliptic if for any {\it real} 1-form $\omega$ satisfying  $\left\langle \omega, L_{j}\right\rangle=0$ for all $j=1,\dots,n$, we have $\omega=0$.  Consequently, the number $n$ of vector fields must satisfy
$
\frac{N}{2}\le n\le N.
$
Alternatively, the ellipticity of the system is equivalent to saying that the second order operator
\begin{equation}\nonumber
\Delta_{{\L}}\; :=\;L_1^* L_1+\cdots+L_n^* L_n
\end{equation}
is elliptic. As a direct consequence of the previous non-homogeneous version, we extend \cite[Theorem A]{HHP}.
 
 \begin{corollary}
Assume that the system of complex vector fields $\left\{L_{1},\dots,L_{n}\right\}$ is elliptic on $\Omega\subset\erre^N$, and let $p,q$ satisfy
\[
\frac{N}{N+1}<p<\infty,\quad 1<q\leq \infty,\quad \frac{1}{r} := \frac{1}{p}+\frac{1}{q}<1+\frac{1}{N}.
\]
Then for every point $x_{0} \in \Omega$, there exist an open neighborhood $x_0\in U \subset \Omega$ and a constant $C>0$
such that 
\begin{equation}\label{eq04}
\|\nabla_{\L}\phi\cdot v\|_{h^r} \le C\|\nabla_{\L}\phi\|_{h^p} \left( \|v\|_{h^{q}}+\|{\rm div}_{\L^{*}}v\|_{h^{q}} \right),
\end{equation}
holds for any $\phi\in C_{c}^{\infty}(U)$ and $v\in C^{\infty}_{c}(U,{\C^{n}})$. 
\end{corollary}

\subsection{Elliptic complexes of vector fields }\label{aplica2}

Let  $C^{\infty}(\Omega,\Lambda^{k}\R^{n})$ denote the space of $k$-forms on $\R^{n}$, $0\leq k \leq n$, with smooth complex coefficients defined on $\Omega$. Each $f \in C^{\infty}(\Omega,\Lambda^{k}\R^{n})$ may be written as
$
\displaystyle{f=\sum_{|I|=k}f_{I}dx_{I}}$, with $dx_{I}=dx_{i_1}\wedge\cdots\wedge dx_{i_k},
$
where one has $f_{I}\in C^{\infty}(\Omega)$ and $I=\left\{i_{1},...,i_{k}\right\}$ is a set of strictly increasing indices with $i_{l}\in \left\{1,...,n\right\}$,  $l=1,...,k$.  Consider the differential operators
\[
d_{\bsL,k}: C^{\infty}(\Omega,\Lambda^{k}\R^{n})\rightarrow C^{\infty}(\Omega,\Lambda^{k+1}\R^{n})
\]
defined by
$$d_{\bsL,0}f:=\sum_{j=1}^{n}(L_{j}f)dx_{j}$$ for $f \in C^{\infty}(\Omega)$,
and
\begin{equation}\nonumber
d_{\bsL,k}f:=\sum_{|I|=k}(d_{\bsL,0}f_{I})dx_{I}=\sum_{|I|=k}\sum_{j=1}^{n}(L_{j}f_{I})dx_{j} \wedge dx_{I}
\end{equation}  
 for $f=\sum_{|I|=k}f_{I}dx_{I}\in C^{\infty}(\Omega,\Lambda^{k}\R^{n})$ and $1\le k\le n-1$.
We also define the dual pseudo-complex $d^{*}_{\bsL,k}: C^{\infty}(\Omega,\Lambda^{k+1}\R^{n})\rightarrow 
C^{\infty}(\Omega,\Lambda^{k}\R^{n})$, $0 \leq k \leq n-1$, determined {by the following relation for any $u \in C_{c}^{\infty}(\Omega,\Lambda^{k}\R^{n})$ and $v \in C_{c}^{\infty}(\Omega,\Lambda^{k+1}\R^{n})$:}
$$\int d_{\bsL,k}u\cdot \overline{v}\; =\int u \cdot \overline{d^{*}_{\bsL,k}v},$$
where the dot indicates the standard pairing on forms of the same degree. Explicitly, for  
$\displaystyle{f=\sum_{|J|=k}f_{J}dx_{J}}$, we have
$$\displaystyle{d_{\bsL,k}^{*}f=\sum_{|J|=k}\sum_{j \in J}L_{j}^{*}f_{J}dx_{j} \vee dx_{J}},$$
where, for each  $j_{l} \in J=\left\{j_{1},...,j_{k}\right\}$ and  $l \in \left\{1,...,k\right\}$, $dx_{j_{l}} \vee dx_{J}$ is defined by:
$
dx_{j_{l}} \vee dx_{J}:=(-1)^{l+1}dx_{1} \wedge ... \wedge dx_{j_{l-1}}\wedge dx_{j_{l+1}}\wedge ... \wedge dx_{j_{k}}.
$

Now suppose that the system $\bsL$ is involutive, \textit{i.e.} that each commutator $[L_{j},L_{\ell}]$ for $1\leq j,\ell \leq n$ is a linear combination of $L_{1},\dots,L_{n}$. 
Then the chain $\{ d_{\bsL,k} \}_{k}$  defines a complex of differential operators associated to the structure $\bsL$, which is precisely the \textit{de Rham complex} when $n=N$ and $L_{j}=\partial_{x_{j}}$  (see \cite{BCH} for more details). In the non-involutive situation, we do not get a complex in general, and the fundamental complex property $d_{\bsL,k+1}\circ d_{\bsL,k}=0$ might not hold. On the other hand, this chain still satisfies a ``pseudo-complex" property in the sense that
$d_{\bsL,k+1}\circ d_{\bsL,k}$ is a differential of operator of order one,
 rather than order two as generically expected. We refer to $(d_{\bsL,k},C^{\infty}(\Omega,\Lambda^{k}\R^{n}))$ as the pseudo-complex $\{d_{\bsL}\}$ associated with $\bsL$ on $\Omega$. 

Consider the operator $$A(\,\cdot \,,D)=(d_{\bsL,k}, d^{*}_{\bsL,k-1}): C_{c}^{\infty}(\Omega,\Lambda^{k}\R^{n}) \rightarrow C_{c}^{\infty}(\Omega,\Lambda^{k+1}\R^{n}) \times C_{c}^{\infty}(\Omega,\Lambda^{k-1}\R^{n}),$$
for $0\leq k \leq n$. Here the operator $d_{\bsL,-1}=d^*_{\bsL,-1} $ is understood to be zero. From \cite[Lemma 3.1]{HP2} the operator $A(\cdot,D)$ is elliptic. It is clear that 
for $(u,v) \in C_{c}^{\infty}(\Omega,\Lambda^{k+1}\R^{n}) \times C_{c}^{\infty}(\Omega,\Lambda^{k-1}\R^{n})$
we have
\begin{equation}\label{aplic1}
A^{*}(\cdot,D)(u,v)=d^{*}_{\bsL,k} u + d_{\bsL,k-1} v.
\end{equation}

 

\begin{corollary}
Assume that the system of vector fields $\left\{L_{1},\dots,L_{n}\right\}$ is elliptic, let $0 \leq k \leq n$, and $p,q,r$ be as in Theorem  A. Then every point $x_{0} \in \Omega$ is contained in an open neighborhood $U \subset \Omega$ such that the estimate 
\begin{equation*}
\|d_{\L,k}\phi \cdot v_{1}+ d_{\L,k}\phi \cdot v_{2}\|_{h^{r}}\leq C \left(\|d_{\L,k}\phi\|_{h^{p}}+\|d^{*}_{\L,k-1}\phi\|_{h^{p}} \right)\left(\|v\|_{h^{q}}+ \|d^{*}_{\bsL,k} v_{1}\|_{h^{q}} + \|d_{\bsL,k-1} v_{2}\|_{h^{q}}\right).
\end{equation*}
 holds for some $C>0$ and for every $\phi \in C_{c}^{\infty}(U,\Lambda^{k}\R^{n})$, $v_{1} \in C_{c}^{\infty}(U,\Lambda^{k+1}\R^{n})$ and $v_{2} \in C_{c}^{\infty}(U,\Lambda^{k-1}\R^{n})$.
\end{corollary}

In particular taking $v_{2}=0$ and $v=v_{1}$, we obtain 
\begin{equation*}
\|d_{\L,k}\phi \cdot v\|_{h^{r}}\leq C \left(\|d_{\L,k}\phi\|_{h^{p}}+\|d^{*}_{\L,k-1}\phi\|_{h^{p}} \right)\left(\|v\|_{h^{q}}+ \|d^{*}_{\bsL,k} v\|_{h^{q}} \right)
\end{equation*}
for all $\phi \in C_{c}^{\infty}(U,\Lambda^{k}\R^{n})$ and $v \in C_{c}^{\infty}(U,\Lambda^{k+1}\R^{n})$,
which recovers Theorems B and C in \cite{HHP}; in the latter we assume that the system $\L$ is involutive.  
   
 \vglue 1cm  

\noindent \textbf{Acknowledgments:} The authors thank Jorge Hounie, Gabriel Ara\'ujo, and Gustavo Hoepfner for their helpful suggestions and contributions. 


\begin{thebibliography}{CLMS}
	

\bibitem{AP}{M. Almeida and T. Picon}, 
\textit{Fourier transform decay of distributions in Hardy-Morrey spaces}, 
Results in Mathematics \textbf{79}, 104 (2024), pp. 1--24.

\bibitem{AH} J.{ \'Alvarez and J. Hounie}, 
\textit{Estimates for the kernel and continuity  properties of pseudodifferential operators, }
Ark. f\"or Mat. \textbf{28} (1990), 1--22.

\bibitem{ART} {P. Auscher, E. Russ, and P. Tchamitchian}, \textit{Hardy-Sobolev spaces on strongly Lipschitz domains of $\R^{N}$  }, J. Funct. Anal. \textbf{218} (2005), 54--109.


\bibitem{Beals} R. Beals, \textit{$L^p$ and H\"older estimates for pseudodifferential operators: sufficient conditions},
Annales de l'institut Fourier,  \textbf{29}, no 3 (1979), p. 239--260.

\bibitem{BCH} {S. Berhanu, P. Cordaro, and J. Hounie}, {\it An Introduction to Involutive Structures}, Cambridge University Press  (2008).


\bibitem{BFG} {A. Bonami, J. Feuto, and S. Grellier}, {\it Endpoint for the div-curl lemma in Hardy spaces}, Publ. Math. \textbf{54} (2010), no. 2, 341--358. 

\bibitem{BGK} {A. Bonami, S. Grellier, and L. Ky}, {\it Paraproducts and products of function in $BMO(\R^{N})$ and $H^{1}(\R^{N})$ through wavelets}, J. Math. Pures Appl. \textbf{97} (2012), no. 3, 230--241. 

\bibitem{BIJZ} {A. Bonami, T. Iwaniec, P. Jones, and M. Zinsmeister}, {\it On the product of functions in $BMO$ and $H^{1}$}, Ann. Inst. Fourier, Grenoble \textbf{57} (2007) 1405--1439. 

\bibitem{B} {M. Bownik}, \textit{Boundedness of operators on Hardy spaces via atomic decompositions}. Proc. Amer. Math. Soc. \textbf{133} (2005), no. 12, 3535--3542.

\bibitem{CDY} {C. Chang, G. Dafni and H. Yue}, {\it A div-curl decomposition for the local Hardy space}, Proc. Amer. Math. Soc. \textbf{137} (2009), no. 10, 3369--3377.

\bibitem{CGS} {C. Chang, G. Dafni and C. Sadosky}, {\it A div-curl lemma in BMO on a domain}, Harmonic analysis, signal processing, and complexity, Progr. Math. \textbf{238}, Birkhauser, Boston, MA, (2005), 55--65.

\bibitem{CLMS} {R. Coifman, P. Lions,  L. Meyer, and S. Semmes}, {\it Compensated Compactness and Hardy Spaces}, J. Math. Pures Appl. \textbf{72} (1993), no. 9, 247--286. 

\bibitem{D} {G. Dafni}, {\it Nonhomogeneus div-curl Lemmas and Local Hardy Spaces}, Adv. in Diff. Eq. \textbf{10} (2005), no. 5, 505--526. 

\bibitem{D1} G. Dafni, C.H. Lau, T. Picon and C. Vasconcelos, \textit{Necessary cancellation conditions for the 
boundedness of operators on local Hardy spaces}, Potential Analysis \textbf{1} (2023), 1--11.

\bibitem{D2} G. Dafni, C.H. Lau, T. Picon, and C. Vasconcelos, \textit{Inhomogeneous cancellation conditions and Calder\'on-Zygmund type operators on $h^p$}, Nonlinear Analysis-theory Methods \& Applications \textbf{225} (2022), p. 113110.



\bibitem{G} {D. Goldberg}, \textit{A local version of real Hardy spaces,} Duke Math. J. \textbf{46} (1979), 27--42.


\bibitem{HHP} {G. Hoepfner, J. Hounie and T. Picon}, \textit{Div-curl type estimates for elliptic systems of complex vector fields}, Journal of Mathematical Analysis and Applications (Print) \textbf{429} (2015), 774--799.

\bibitem{H} {L. H\"ormander}, \textit{The analysis of linear partial differential operators III}, Springer-Verlag, Berlin-New York (1994).


\bibitem{HP2} {J. Hounie and T. Picon}, \textit{Local $L^{1}$ estimates for elliptic systems of complex vector fields}, Proc. Amer. Math. Soc. \textbf{143} (2015), 1501--1514.



\bibitem{KS} {P. Koskela and E. Saksman}, 
\textit{Pointwise characterization of Hardy-Sobolev functions}, Math. Res. Lett. \textbf{15} (2008), no. 4, 727--744.


\bibitem{KY} H. Kozono and T. Yanagisawa,
\textit{Global compensated compactness theorem for general differential operators of first order}, (English summary) Arch. Ration. Mech. Anal.207(2013), no.3, 879-905.


\bibitem{LM1} {Z. Lou and A. McIntosh}, 
\textit{Hardy spaces of exact forms on $\R^{N}$}, Trans. Amer. Math. Soc. \textbf{357} (2005), no. 4, 1469--1496.

\bibitem{LM2} {Z. Lou and A. McIntosh}, 
\textit{Hardy spaces of exact forms on Lipschitz domains in $\R^{N}$}, Indiana Univ. Math. J. \textbf{53} (2004), no. 4, 583--611.

\bibitem{MSV} S. Meda, P.  Sj\"ogreen, and M. Vallarino, \textit{On the $H^1$--$L^1$ boundedness of operators},
 Proc. Amer. Math. Soc. \textbf{136} (2008), no. 8, 2921-2931.

\bibitem{Miyakawa} T. Miyakawa, \textit{Hardy spaces for solenoidal vector fields, with applications to the Navier Stokes equations}, Kyushu J. Math,  vol 50 (1996), pp. 1-64.
 

\bibitem{Mi}{ A. Miyachi}, 
\textit{Hardy-Sobolev spaces and maximal functions}, J. Math. Soc. Japan \textbf{42} (1990), no. 1, 73--90.

\bibitem{Mi2}{A. Miyachi}, 
\textit{Maximal functions for distributions on open sets}, Hitotsubashi J. Arts Sci. \textbf{28} (1987), 45--58.



\bibitem{St} {E. Stein},  \textit{Harmonic Analysis: real-variable methods, orthogonality, and oscillatory integrals}, Princeton University Press, New Jersey, (1993).

\bibitem{St2} {E. Stein},  \textit{Singular Integrals and Differentiability properties of functions}, Princeton University Press, New Jersey, (1970).


\bibitem{SW1}  {E. Stein and G. Weiss},   \textit{Introduction to Fourier Analysis on Euclidean Spaces}, Princeton University Press, New Jersey, (1975). 


\bibitem{Tay} M. Taylor, 
\textit{Tools for PDE}, American Math. Society, Providence, (2000).

\bibitem{Tri1} {H. Triebel},  \textit{Theory of function spaces}, Monographs in Math. \textbf{78}, Birk\-hauser, (1983).



\bibitem{TrPseud} {F. Treves}, 
	\textit{Introduction to Pseudodifferential and Fourier Integral Operators: Pseudodifferential Operators}, 
	University Series in Mathematics, Springer, New York (1980).

\bibitem{YYZ} D. Yang, W. Yuan, and Y. Zhang, \textit{Bilinear decomposition and divergence-curl estimates on products related to local Hardy spaces and their dual spaces}, J. Funct. Anal. 280 (2021), no.2, paper no. 108796, 74 pp.  

\end{thebibliography}
\end{document}